\documentclass[11pt,a4paper]{article}

\usepackage{colortbl} 
\usepackage{amsmath}
\usepackage{mhequ}
\usepackage{booktabs}
\usepackage{centernot}
\usepackage{amsthm}
\usepackage{amssymb}
\usepackage{mathrsfs}
\usepackage{enumerate} 
\usepackage{marginnote}
\usepackage{scalerel}
\usepackage{tikz}
\usepackage{bbm}
\usepackage{tikz-cd}
\usepackage{graphicx}
\usepackage{caption}
\usepackage{subcaption}
\usepackage{pdfpages}

\numberwithin{equation}{section}

\newtheorem{thm}{Theorem}[section]

\newtheorem*{thm*}{Theorem}
\newtheorem*{prop*}{Proposition}
\newtheorem*{lem*}{Lemma}
\newtheorem{lemma}{Lemma}

\newtheorem{lem}[thm]{Lemma}
\newtheorem{prop}[thm]{Proposition}
\newtheorem{cor}[thm]{Corollary}
\newtheorem{defn}[thm]{Definition}
\newtheorem{rem}[thm]{Remark}

\newcommand{\NN}{\mathbb{N}}

\newcommand{\RR}{\mathbb{R}}

\newcommand{\cB}{\mathcal B}

\newcommand{\cF}{\mathcal F}

\newcommand{\cH}{\mathcal H}

\newcommand{\cM}{\mathcal M}

\newcommand{\cT}{\mathcal T}
\newcommand{\ct}{\mathcal{t}}

\makeatletter
\newsavebox{\@brx}
\newcommand{\llangle}[1][]{\savebox{\@brx}{\(\m@th{#1\langle}\)}%
  \mathopen{\copy\@brx\kern-0.5\wd\@brx\usebox{\@brx}}}
\newcommand{\rrangle}[1][]{\savebox{\@brx}{\(\m@th{#1\rangle}\)}%
  \mathclose{\copy\@brx\kern-0.5\wd\@brx\usebox{\@brx}}}
\makeatother

\colorlet{symbols}{black}
\colorlet{testcolor}{green!60!black}

\def\1{\mathbf{{1}}}

\usetikzlibrary{shapes.misc}
\usetikzlibrary{shapes.symbols}
\usetikzlibrary{snakes}
\usetikzlibrary{decorations}
\usetikzlibrary{decorations.markings}
\definecolor{dblue}{rgb}{0.1, 0.1, 0.9}

\tikzset{
	root/.style={circle,fill=testcolor,inner sep=0pt, minimum size=2mm},		
	dot/.style={circle,fill=black,draw=black, solid,inner sep=0pt,minimum size=0.75mm},
	bdot/.style={circle,fill=blue,draw=dblue, solid,inner sep=0pt,minimum size=0.75mm},
		}
\colorlet{symbols}{blue!90!black}

\makeatletter
\def\DeclareSymbol#1#2#3{\expandafter\gdef\csname MH@symb@#1\endcsname{\tikz[baseline=#2,scale=0.15]{#3}}%
\expandafter\gdef\csname MH@symb@#1s\endcsname{\scalebox{0.6}{\tikz[baseline=#2,scale=0.15]{#3}}}}
\def\<#1>{\csname MH@symb@#1\endcsname}
\makeatother

\DeclareSymbol{sunset}{-3}{\draw (0,0) circle [x radius = 1.5, y radius = 1]; \draw (-1.5,0) node[dot] {} -- (1.5,0) node[dot] {};}
\DeclareSymbol{tadpole}{-3}{\draw (0,0) circle [x radius = 0.75, y radius = 1]; \draw (0,-1) node[dot] {};}

\DeclareSymbol{1}{0}{\draw[white] (-.4,0) -- (.4,0); \draw (0,0)  -- (0,1.2) node[dot] {};}
\DeclareSymbol{2}{0}{\draw (-0.5,1.2) node[dot] {} -- (0,0) -- (0.5,1.2) node[dot] {};}
\DeclareSymbol{3}{0}{\draw (0,0) -- (0,1.2) node[dot] {}; \draw (-.7,1) node[dot] {} -- (0,0) -- (.7,1) node[dot] {};}
\DeclareSymbol{4}{0}{\draw (-0.36,1.2) node[dot] {} -- (0,0) -- (0.36,1.2) node[dot] {}; \draw (-1,1) node[dot] {} -- (0,0) -- (1,1) node[dot] {};}
\DeclareSymbol{30}{-3}{\draw (0,0) -- (0,-1); \draw (0,0) -- (0,1.2) node[dot] {}; \draw (-.7,1) node[dot] {} -- (0,0) -- (.7,1) node[dot] {};}
\DeclareSymbol{31}{-3}{\draw (0,0) -- (0,-1) -- (1,0) node[dot] {}; \draw (0,0) -- (0,1.2) node[dot] {}; \draw (-.7,1) node[dot] {} -- (0,0) -- (.7,1) node[dot] {};}
\DeclareSymbol{32}{-3}{\draw (0,0) -- (0,-1) -- (1,0) node[dot] {}; \draw (0,0) -- (0,-1) -- (-1,0) node[dot] {}; \draw (0,0) -- (0,1.2) node[dot] {}; \draw (-.7,1) node[dot] {} -- (0,0) -- (.7,1) node[dot] {};}
\DeclareSymbol{20}{-3}{\draw (0,0) -- (0,-1);\draw (-.7,1) node[dot] {} -- (0,0) -- (.7,1) node[dot] {};}
\DeclareSymbol{22}{-3}{\draw (0,0.3) -- (0,-1) -- (1,0) node[dot] {}; \draw (0,0.3) -- (0,-1) -- (-1,0) node[dot] {};\draw (-.7,1) node[dot] {} -- (0,0.3) -- (.7,1) node[dot] {};}
\DeclareSymbol{202}{-3}{\draw (-.6625,0) -- (0,-1) -- (.6625,0)  {}; \draw (-1.1,1) node[dot] {} -- (-.6625,0) -- (-0.365,1) node[dot] {}; \draw (0.365,1) node[dot] {} -- (.6625,0) -- (1.1,1) node[dot] {};}

\DeclareSymbol{1b}{0}{\draw[white] (-.4,0) -- (.4,0); \draw[color=dblue] (0,0)  -- (0,1.2) node[bdot] {};}
\DeclareSymbol{2b}{0}{\draw[color=dblue] (-0.5,1.2) node[bdot] {} -- (0,0) -- (0.5,1.2) node[bdot] {};}
\DeclareSymbol{3b}{0}{\draw[color=dblue] (0,0) -- (0,1.2) node[bdot] {}; \draw[color=dblue] (-.7,1) node[bdot] {} -- (0,0) -- (.7,1) node[bdot] {};}
\DeclareSymbol{30b}{-3}{\draw[color=dblue] (0,0) -- (0,-1); \draw[color=dblue] (0,0) -- (0,1.2) node[bdot] {}; \draw[color=dblue] (-.7,1) node[bdot] {} -- (0,0) -- (.7,1) node[bdot] {};}
\DeclareSymbol{31b}{-3}{\draw[color=dblue] (0,0) -- (0,-1) -- (1,0) node[bdot] {}; \draw[color=dblue] (0,0) -- (0,1.2) node[bdot] {}; \draw[color=dblue] (-.7,1) node[bdot] {} -- (0,0) -- (.7,1) node[bdot] {};}
\DeclareSymbol{32b}{-3}{\draw[color=dblue] (0,0) -- (0,-1) -- (1,0) node[bdot] {}; \draw[color=dblue] (0,0) -- (0,-1) -- (-1,0) node[bdot] {}; \draw[color=dblue] (0,0) -- (0,1.2) node[bdot] {}; \draw[color=dblue] (-.7,1) node[bdot] {} -- (0,0) -- (.7,1) node[bdot] {};}
\DeclareSymbol{20b}{-3}{\draw[color=dblue] (0,0) -- (0,-1);\draw[color=dblue] (-.7,1) node[bdot] {} -- (0,0) -- (.7,1) node[bdot] {};}
\DeclareSymbol{22b}{-3}{\draw[color=dblue] (0,0.3) -- (0,-1) -- (1,0) node[bdot] {}; \draw[color=dblue] (0,0.3) -- (0,-1) -- (-1,0) node[bdot] {};\draw[color=dblue] (-.7,1) node[bdot] {} -- (0,0.3) -- (.7,1) node[bdot] {};}

\DeclareSymbol{box}{0.65}{\draw[semithick] (-0.7,0) -- (0.7,0) --  (0.7,1.4) -- (-0.7,1.4) -- (-0.7,0);}
\DeclareSymbol{sbox}{0.5}{\draw[semithick] (-0.5,0) -- (0.5,0) -- (0.5,1) -- (-0.5,1) -- (-0.5,0);}
\DeclareSymbol{blackbox}{0.65}{\draw[semithick, fill=black] (-0.7,0) -- (0.7,0) --  (0.7,1.4) -- (-0.7,1.4) -- (-0.7,0);}
\DeclareSymbol{boxinbox}{0.65}{\draw[semithick, fill=black] (-0.7,0) -- (0.7,0) --  (0.7,1.4) -- (-0.7,1.4) -- (-0.7,0); \draw[semithick, fill=white] (-0.7,0) -- (0,0) -- (0,0.7) -- (-0.7,0.7) -- (-0.7,0) ; }
\DeclareSymbol{sboxinbox}{0.5}{\draw[semithick, fill=black] (-0.5,0) -- (0.5,0) -- (0.5,1) -- (-0.5,1) -- (-0.5,0);\draw[semithick, fill=white] (-0.5,0) -- (0,0) -- (0,0.5) -- (-0.5,0.5) -- (-0.5,0) ; }

\let\P= \undefined
\newcommand{\P}{\mathbf{P}}

\newcommand\n{\mathbf{n}}
\newcommand\m{\mathbf{m}}
\newcommand\sn{\partial\mathbf{n}}
\newcommand\sm{\partial\mathbf{m}}

\usepackage[margin=3cm]{geometry}

\usepackage{amsfonts}
\usepackage{amsthm}
\usepackage{amsmath}
\usepackage{amssymb}
\usepackage{float}
\usepackage[T1]{fontenc}
\usepackage{lmodern}
\usepackage[utf8]{inputenc}
\usepackage[cal=boondoxo]{mathalfa}

\title{Scaling limit of the cluster size distribution for the random current measure on the complete graph}

\begin{document}

\author{Dmitrii Krachun\footnotemark[1]\footnote{Princeton University,~dk9781@princeton.edu}, \, Christoforos Panagiotis\footnotemark[2]\footnote{University of Bath,~cp2324@bath.ac.uk}, \, Romain Panis\footnotemark[3]\footnote{Université de Genève,~romain.panis@unige.ch}}

\maketitle

\begin{abstract}
We study the percolation configuration arising from the random current representation of the near-critical Ising model on the complete graph. We compute the scaling limit of the cluster size distribution for an arbitrary set of sources in the single and the double current measures. As a byproduct, we compute the tangling probabilities recently introduced by Gunaratnam, Panagiotis, Panis, and Severo in \cite{gunaratnam2022random}. This provides a new perspective on the switching lemma for the $\varphi^4$ model introduced in the same paper: in the Gaussian limit we recover Wick's law, while in the Ising limit we recover the corresponding tool for the Ising model.
\end{abstract}

\section{Introduction}
For $n \geq 1$, consider the complete graph $K_n=(V_n,\mathcal{P}_2(V_n))$ on $n$ vertices, where $\mathcal{P}_2(V_n)$ is the set of pairs of elements of $V_n$. Let $\lambda\in \mathbb R$, and define the inverse temperature
\begin{equs}
    d_n(\lambda):=\frac{1}{n}\left(1-\frac{\lambda}{\sqrt{n}}\right).
\end{equs}
For any $\sigma\in \lbrace \pm 1\rbrace^{V_n}$, define the Hamiltonian of the Ising model on the complete graph $K_n$ by
\begin{equs}
    \mathcal{H}_{n,\lambda}(\sigma):=-\sum_{\lbrace i,j\rbrace\in \mathcal{P}_2(V_n)}d_n(\lambda)\sigma_i\sigma_j.
\end{equs}
The associated Ising measure is defined via the expectation values 
\begin{equs}
    \langle F(\sigma)\rangle_{n,\lambda}:=\frac{1}{\mathcal{Z}_n(\lambda)}\sum_{\sigma\in \lbrace \pm 1\rbrace^{V_n}}F(\sigma)\exp(-\mathcal{H}_{n,\lambda}(\sigma))
\end{equs}
for $F: \lbrace \pm 1\rbrace^{V_n}\rightarrow \mathbb R^+$,
where $\mathcal{Z}_n(\lambda)$ is the partition function of the model. This model, which is also known as the \emph{near-critical Curie--Weiss} model, has been introduced independently by many people, including \cite{temperley1954mayer,husimi1953statistical}, and since then, it has been extensively studied. See \cite[Chapter~2]{friedli2017statistical} and references therein. 

We are interested in a particular expansion of the partition function of the model called the \emph{random current expansion} which provides an interesting geometric representation of the correlation functions of the model. This expansion was initially introduced in \cite{GHS} to prove concavity of the magnetisation with a positive external field, but the probabilistic picture attached to it was developed in \cite{A}. We refer to \cite{duminil2017lectures} for more information about this expansion.

A \emph{current} on $K_n$ is a function $\n:\mathcal{P}_2(V_n)\rightarrow ~\mathbb N:=\lbrace 0,1,\ldots\rbrace$. For $\lbrace i,j\rbrace\in \mathcal{P}_2(V_n)$ and $\n\in \Omega_n$, we will write $\n_{i,j}=\n(i,j)$ for the value of $\n$ on the edge $\lbrace i,j\rbrace$. We denote by $\Omega_n$ the set of currents on $K_n$. The set of \emph{sources} of $\n$, denoted by $\sn$, is defined as
\begin{equs}
\sn:=\left\lbrace x \in V_n \: : \: \sum_{y\in V_n\setminus\{x\}}\n_{x,y}\textup{ is odd}\right\rbrace.
\end{equs}
It is an easy exercise to check that $|\sn|$ is always even. For a current $\n$ on $K_n$, introduce the weight
\begin{equs}
    w(\n)=w_{n,\lambda}(\n):=\prod_{\lbrace i,j\rbrace\in \mathcal{P}_2(V_n)} \frac{d_n(\lambda)^{\n_{i,j}}}{\n_{i,j}!}.
\end{equs}
For $S$ an even subset of $V_n$ (i.e.\ of even cardinality), let $Z_n^S(\lambda):=\sum_{\sn=S}w(\n)$. When $S=\emptyset$, we let $Z_n^\emptyset(\lambda)=Z_n(\lambda)$. It is classical \cite{GHS,A} that 
\begin{equs}\label{eq: link between partition functions}
    \mathcal{Z}_n(\lambda)=2^nZ_n(\lambda),
\end{equs}
and furthermore, if $\sigma_S:= \prod_{j\in S}\sigma_j$,
\begin{equs}
    \langle \sigma_S\rangle_{n,\lambda}=\frac{Z_n^S(\lambda)}{Z_n^\emptyset(\lambda)}.
\end{equs}

To each current $\n$, we can associate a percolation configuration by looking at the trace of the edges that carry a non-zero weight, i.e. $(\mathbbm{1}[\n_{i,j}>0])_{\{i,j\}\in \mathcal{P}_2(V_n)}$. As is standard in percolation theory, our interest lies in the connectivity properties of this percolation configuration. Let us import some standard terminology. If $x,y \in V_n$, we say that $x$ is \emph{connected} to $y$ in $\n$ and write $x\overset{\n}{\longleftrightarrow} y$, if there exists a sequence of points $x_0=x,x_1,\ldots, x_m=y$ such that $\n_{x_i,x_{i+1}}>0$ for $0\leq i \leq m-1$. For $x\in V_n$, we let $\mathcal{C}_x(\n)$ be the set of vertices of $K_n$ which are connected to $x$ in $\n$. Note that for any $x\in V_n$, $\mathcal{C}_x(\n)$ contains an even number of sources.

We naturally define a probability measure $\mathbb P_{n,\lambda}^S$ on the set of currents with source set $S$ by letting 
\begin{equs}
    \mathbb P_{n,\lambda}^S[\n]:=\frac{1}{Z_n^S(\lambda)}w(\n)\mathbbm{1}[\sn=S].
\end{equs}

The main interest in the current expansion of the Ising model arises from the existence of a combinatorial tool called the \emph{switching lemma}, which was initially introduced in \cite{GHS} and used in many subsequent works to capture fine properties of the Ising model \cite{A,aizenman1983renormalized,aizenman1986critical,aizenman1987phase,DCT,aizenman2019emergent,raoufi2020translation,ADC,panis2023triviality}, see \cite{duminil2017lectures} for more information. This tool relates ratios and differences of correlation functions to a system of two independent currents. This motivates us to also consider, for $S_1,S_2\subset V_n$, the measure 
\begin{equs}\label{eq: switching intro}
    \mathbb P_{n,\lambda}^{S_1,S_2}:=\mathbb P_{n,\lambda}^{S_1}\otimes \mathbb P_{n,\lambda}^{S_2}.
\end{equs}
Applying the switching lemma \cite[Lemma~4.3]{duminil2017lectures}, we find that for any event $\mathcal{A}$ which depends on $\n_1+\n_2$,
\begin{equs}\label{eq: switching 1}
    \mathbb P_{n,\lambda}^{S_1,S_2}[\mathcal{A}]=\frac{\langle \sigma_{S_1\Delta S_2}\rangle_{n,\lambda}}{\langle \sigma_{S_1}\rangle_{n,\lambda}\langle \sigma_{S_2}\rangle_{n,\lambda}} \mathbb P_{n,\lambda}^{S_1 \Delta S_2,\emptyset}[\mathcal{A}, \:\mathcal{F}_{S_2}],
\end{equs}
where $\mathcal{F}_{S_2}$ is the event that there exists a sub-current $\m\leq \n_1+\n_2$ satisfying $\sm=S_2$. Note that if $S_2=\lbrace i,j\rbrace$, $\mathcal{F}_{S_2}$ coincides with the event that $i\overset{\n}{\longleftrightarrow} j$.
This observation and Lemma \ref{lem: GS} (which provides asymptotics as $n\rightarrow \infty$ for the ratio of correlation functions appearing in the right-hand side of \eqref{eq: switching 1}) reduce the analysis to the case $S_2=\emptyset$.

Let $S$ be an even subset of $V_n$. Let $\mathcal{P}_{\mathrm{even}}(S)$ be the set of even partitions of $S$, i.e. the set made of partitions of $S$ consisting only of even subsets of $S$. If $\n$ is a current on $K_n$ with sources $S$, we denote by $\mathfrak{P}(\n)$ the partition of $S$ in the associated percolation configuration. Note that for a current $\n$ one always has that $\mathfrak{P}(\n)\in \mathcal{P}_{\mathrm{even}}(\sn)$.

In \cite[Proposition 4.3]{gunaratnam2022random}, the authors proved\footnote{This observation was already made in \cite{A}.} that the clusters of the sources have size of order $\sqrt{n}$. Our main result strengthens this observation. For simplicity we focus on the clusters of the sources. Beyond these clusters, there could be several other clusters of size of order $\sqrt{n}$. Their study reduces to the case below as explained in Remark \ref{rem: study of the clusters of non-sources}.

\begin{thm}\label{thm: main theorem} Let $\lambda\in \mathbb R$ and $k\geq 1$. Set $S=\lbrace 1,\ldots, 2k\rbrace$. Let $\n_1$ (resp. $\n_2$) be distributed according to $\mathbb P_{n,\lambda}^S$ (resp. $\mathbb P^\emptyset_{n,\lambda}$). There exist two random variables $X_{k,\lambda}^{\mathsf{s}},X_{k,\lambda}^{\mathsf{d}}$ on $(\mathbb R_{\geq 0})^{2k}\times \mathcal{P}_{\mathrm{even}}(S)$ such that
\begin{equs}
    \left(\frac{|\mathcal{C}_1(\n_1)|}{\sqrt{n}},\ldots, \frac{|\mathcal{C}_{2k}(\n_1)|}{\sqrt{n}}, \mathfrak{P}(\n_1)\right)\underset{n \rightarrow \infty}{\longrightarrow}X_{k,\lambda}^{\mathsf{s}},
\end{equs}
\begin{equs}
    \left(\frac{|\mathcal{C}_1(\n_1+\n_2)|}{\sqrt{n}},\ldots, \frac{|\mathcal{C}_{2k}(\n_1+\n_2)|}{\sqrt{n}}, \mathfrak{P}(\n_1+\n_2)\right)\underset{n \rightarrow \infty}{\longrightarrow}X_{k,\lambda}^{\mathsf{d}},
\end{equs}
where both convergences are in distribution. Moreover, each of the first $2k$ coordinates of $X_{k,\lambda}^{\mathsf{s}}$ (resp. $X_{k,\lambda}^{\mathsf{d}}$) is strictly positive almost surely.
\end{thm}
\begin{rem} In Section \textup{\ref{section: scaling limit of cluster size distribution}}, we will obtain exact formulas for the density of these random variables (see Proposition \textup{\ref{prop: main prop}}).
\end{rem}

To prove Theorem \ref{thm: main theorem}, we will first analyse the structure of the clusters of the sources in the limit $n\rightarrow \infty$, see Section \ref{section: properties of the current}. As it turns out, the clusters have a simple combinatorial description in terms of a family of mutligraphs we call \emph{backbone multigraphs} which in particular carry the information of how the sources are connected together, see Definition \ref{def: backbone multigraph}. Conditionally on this backbone multigraph, the scaling limit of the cluster size distribution can be computed exactly, see \eqref{eq: eq limit single} and \eqref{eq: eq limit double}.

In comparison with our model, for the well-studied (near-)critical \footnote{The critical threshold for the Erd\H{o}s--Rényi model on $K_n$ is $\tfrac{1}{n}$. The near-critical regime is given by a parameter of the form $\tfrac{1}{n}\Big(1-\tfrac{\lambda}{n^{1/3}}\Big)$ with $\lambda \in \mathbb R$.} Erd\H{o}s--R{\'e}nyi model, the size of the largest cluster is of order $n^{2/3}$. The asymptotic joint distribution of components sizes after normalising by $n^{-2/3}$ was obtained by Aldous \cite{aldous1997brownian}, which lead to the continuum limit of Erd\H{o}s--R{\'e}nyi components \cite{addario2012continuum}, as well as to results in the spirit of Theorem~\ref{thm: main theorem} for several other models, see e.g.\ \cite{bollobas1996random,aldous2000random,nachmias2010critical}. See also \cite{van2016random} and references therein.

As a corollary of Theorem \ref{thm: main theorem} and Proposition \ref{prop: gaussian}, where we obtain the scaling limit of the total number of vertices of positive degree, we get the following result. Recall that a random variable $X$ on $\mathbb R^d$ is said to have Gaussian tails if there exists $c>0$ such that for all $t\geq 1$,
\begin{equs}
    \mathbb P[|X|\geq t]\leq \frac{1}{c}e^{-ct^2},
\end{equs}
where $|\cdot|$ denotes the Euclidean norm on $\mathbb R^d$.
\begin{cor}[Scaling limit of the cluster size distribution]\label{cor: cor 1} Let $\lambda\in \mathbb R$ and $k\geq 1$. Set $S=\lbrace 1,\ldots, 2k\rbrace$. Let $\n_1$ (resp. $\n_2$) be distributed according to $\mathbb P_{n,\lambda}^S$ (resp. $\mathbb P^\emptyset_{n,\lambda}$). Let 
$Y_{k,\lambda}^{\mathsf{s}}$ (resp. $Y_{k,\lambda}^{\mathsf{d}}$) be the projection on the first $2k$ coordinates of $X_{k,\lambda}^{\mathsf{s}}$ (resp. $X_{k,\lambda}^{\mathsf{d}})$. Then, $Y_{k,\lambda}^{\mathsf{s}}, Y_{k,\lambda}^{\mathsf{d}}$ have Gaussian tails and,
\begin{equs}\label{eq: equation 1 cor 1}
    \left(\frac{|\mathcal{C}_1(\n_1)|}{\sqrt{n}},\ldots, \frac{|\mathcal{C}_{2k}(\n_1)|}{\sqrt{n}}\right)\underset{n \rightarrow\infty}{\longrightarrow}Y_{k,\lambda}^{\mathsf{s}},
\end{equs}
\begin{equs}\label{eq: equation 2 cor 1}
    \left(\frac{|\mathcal{C}_1(\n_1+\n_2)|}{\sqrt{n}},\ldots, \frac{|\mathcal{C}_{2k}(\n_1+\n_2)|}{\sqrt{n}}\right)\underset{n \rightarrow\infty}{\longrightarrow}Y_{k,\lambda}^{\mathsf{d}},
\end{equs}
where both convergences are in distribution.
\end{cor}

One striking property of the near-critical Ising model introduced above is that its scaling limit is (under appropriate scaling) the $\varphi^4$ measure, see \cite{simon1973varphi} or Appendix \ref{appendix: proof of GS}. This observation plays a central role in \cite{gunaratnam2022random} where a switching lemma for the discrete $\varphi^4$ model (see Section \ref{section: phi4} for a definition) is obtained by approximating the latter model with a well-chosen sequence of Ising models of the above type. As it turns out, the limiting objects are related to the distribution of $\mathfrak{P}(\n)$. In the latter paper, even partitions of $S$ are referred to as \emph{tanglings}, see Section \ref{section: phi4}.

In the proof of Theorem \ref{thm: main theorem}, we will also obtain that the projections of $X_{k,\lambda}^{\mathsf{s}}$ and $X_{k,\lambda}^{\mathsf{d}}$ on $\mathcal{P}_{\mathrm{even}}(S)$ give positive probability to any even partition. This yields the following result which in particular improves on \cite[Lemma~4.7]{gunaratnam2022random}, where the full convergence of the tangling probabilities is missing.

\begin{cor}[Convergence of the tanglings probabilities]\label{cor: cor 2}
Let $\lambda \in \mathbb R$ and $k\geq 1$. Set $S=\lbrace 1,\ldots, 2k\rbrace$. Fix $\mathsf{P}\in \mathcal{P}_{\mathrm{even}}(S)$. Let $\n_1$ (resp. $\n_2$) be distributed according to $\mathbb P_{n,\lambda}^S$ (resp. $\mathbb P^\emptyset_{n,\lambda}$). Let 
$P_{k,\lambda}^{\mathsf{s}}$ (resp. $P_{k,\lambda}^{\mathsf{d}}$) be the projection on $\mathcal{P}_{\mathrm{even}}(S)$ of $X_{k,\lambda}^{\mathsf{s}}$ (resp. $X_{k,\lambda}^{\mathsf{d}})$. Then,
\begin{equs}
    \mathfrak{P}(\n_1)\underset{n \rightarrow\infty}{\longrightarrow}P_{k,\lambda}^{\mathsf{s}},
\end{equs}
\begin{equs}
    \mathfrak{P}(\n_1+\n_2)\underset{n \rightarrow\infty}{\longrightarrow}P_{k,\lambda}^{\mathsf{d}},
\end{equs}
where both convergences are in distribution. Moreover, the support of $P_{k,\lambda}^{\mathsf{s}}$ and $P_{k,\lambda}^{\mathsf{d}}$ is equal to $\mathcal{P}_{\mathrm{even}}(S)$.
\end{cor}
As mentioned above, this last corollary has interesting consequences regarding the study of the discrete $\varphi^4$ model. In particular, it will be used to provide a new perspective on the switching lemma developed in \cite{gunaratnam2022random}, by showing that this combinatorial tool interpolates between the corresponding tool for Ising (as $\lambda$ tends to $-\infty$) and Wick's law (as $\lambda$ tends to $+\infty$). See Section \ref{section: phi4} for more details. 

\subsubsection*{Notations}

We will use Landau's formalism: if $(u_n)_{n\geq 0}, (v_n)_{n\geq 0}\in (\mathbb R_+^*)^{\mathbb N}$, write $u_n=O(v_n)$ (resp. $u_n=o(v_n)$) if there exists $C>0$ such that for all $n\geq 1$, $u_n\leq C v_n$ (resp. $\lim_{n\rightarrow \infty} u_n/v_n=0$).

\subsubsection*{Acknowledgements} We warmly thank Hugo Duminil-Copin, Trishen Gunaratnam, and Franco Severo for inspiring discussions. We thank an anonymous referee for useful comments. The project was initiated when D.K and C.P were still at the University of Geneva. This research is supported by the Swiss National Science Foundation and
the NCCR SwissMAP.
\section{The random current model on the complete graph.}\label{section: properties of the current}

We begin by importing a result from \cite{gunaratnam2022random}. For $g>0$ and $a\in \mathbb R$, let $\rho_{g,a}$ be the probability measure on $\RR$ defined by
\begin{equs}
    \mathrm{d}\rho_{g,a}(t)
    =
    \frac{1}{z_{g,a}}e^{-gt^4-at^2}\mathrm{d}t,
\end{equs}
where $z_{g,a}=\int_\RR e^{-gt^4-at^2}\mathrm{d}t$. If $g>0$ and $\lambda\in \mathbb R$, let $\tilde g = (12g)^{1/4}$ and choose $a=a(g,\lambda)$ such that $\lambda =2 a {\tilde g}^{-2}$. We now fix $g>0$ for the rest of the paper.
Let 
\begin{equs}\label{eq: def c_N and d_N}
c_n:=\tilde{g}^{-1}n^{-3/4},
\end{equs}
so that $(c_nn)^2=\tilde{g}^{-2}\sqrt{n}$. The following result is a consequence of \cite{simon1973varphi} which states that the scaling limit of the Ising model under consideration is the $\varphi^4$ single-site measure. We reprove this result for sake of completeness.
\begin{lemma}[{\cite[Lemma~3.24]{gunaratnam2022random}}]\label{lem: GS} Let $g>0$ and $\lambda \in \mathbb R$. For any $p\in \mathbb N$,
\begin{equs}
    (c_n n)^p\left\langle \prod_{j=1}^p\sigma_j\right\rangle_{n,\lambda}\underset{n\rightarrow \infty}\longrightarrow \left\langle\varphi^p\right\rangle_0,
\end{equs}
where $\langle \cdot \rangle_0=\langle \cdot \rangle_{0,g,a(\lambda)}$ is the average with respect to the probability measure $\rho_{g,a(\lambda)}$.
\end{lemma}
\begin{proof}
For $p=1$, there is nothing to show because both $\left\langle \sigma_1\right\rangle_{n,\lambda}$ and $\left\langle\varphi\right\rangle_0$ are equal to $0$, so let us assume that $p\geq 2$.

Applying Proposition~\ref{prop: gs approx} for $\lambda=2a\tilde{g}^{-2}$ and doing the change of variables $s=\tilde{g}^2 u$ we get that for every $k\geq 0$
\begin{equation*}
    c_n^k\left\langle \left(\sum_{i=1}^n\sigma_i\right)^k\right\rangle_{n,\lambda}\underset{n\rightarrow \infty}\longrightarrow \langle\varphi^k\rangle_0.
\end{equation*}
Let $n\geq p$. Expanding $\left(\sum_{i=1}^n\sigma_i\right)^p$ and taking cases according to whether some index repeats or not we obtain
\begin{equs}
\left\langle \left(\sum_{i=1}^n\sigma_i\right)^p\right\rangle_{n,\lambda}=\sum_{\substack{(i_1,\ldots,i_p)\in V_n^p\\ \forall k\neq \ell, \: i_k\neq i_\ell}}\left\langle \prod_{j=1}^p\sigma_{i_j}\right\rangle_{n,\lambda} +\sum_{\substack{(i_1,\ldots,i_p)\in V_n^p\\ \exists k<\ell, \: i_k=i_\ell}}\left\langle \prod_{j=1}^p\sigma_{i_j}\right\rangle_{n,\lambda}.
\end{equs}
Note that
\begin{equs}
\sum_{\substack{(i_1,\ldots,i_p)\in V_n^p\\ \forall k\neq \ell, \: i_k\neq i_\ell}}\left\langle \prod_{j=1}^p\sigma_{i_j}\right\rangle_{n,\lambda} =p!\binom{n}{p}\left\langle \prod_{j=1}^p\sigma_j\right\rangle_{n,\lambda} = (1+o(1))n^p\left\langle \prod_{j=1}^p\sigma_j\right\rangle_{n,\lambda}
\end{equs}
by the symmetries of the complete graph. Furthermore,
\begin{equs}
0\leq \sum_{\substack{(i_1,\ldots,i_p)\in V_n^p\\ \exists k<\ell, \: i_k=i_\ell}}\left\langle \prod_{j=1}^p\sigma_{i_j}\right\rangle_{n,\lambda}\leq {p \choose 2}  n \left\langle \left(\sum_{i=1}^n\sigma_i\right)^{p-2}\right\rangle_{n,\lambda}\underset{n\rightarrow \infty}{\longrightarrow} 0
\end{equs}
by choosing $2$ indices $k,\ell$ for which $i_k=i_{\ell}$, and the element of $V_n$ that they are equal.
Here we have also used that each summand is non-negative by Griffiths' inequality. 
\end{proof}
We now turn to some useful estimates on the geometry of the clusters under the measures $\mathbb P_{n,\lambda}^S$ and $\mathbb P^{S,\emptyset}_{n,\lambda}$. Below, we let $k\geq 1$ and $S=\lbrace 1,\ldots,2k\rbrace$. The constant $C$ may depend on $\lambda$. We let $\mathbb P$ denote either $\mathbb P^{S}_{n,\lambda}$ or $\mathbb P^{S,\emptyset}_{n,\lambda}$, and we write $\mathbb E$ for the expectation with respect to $\mathbb P$.
\begin{lem}[Multiplicity of the edges]\label{lem: multiplicativity edges} There exists $C>0$ such that for all $n$ large enough:
\begin{enumerate}
    \item[$(i)$] $\mathbb E\left[\sum_{1\leq i<j\leq n}\mathbbm 1 [\n_{i,j}\geq 1]\right]\leq C\sqrt{n},$
    \item[$(ii)$] $\mathbb E\left[\sum_{1\leq i<j\leq n}\mathbbm 1 [\n_{i,j}\geq 2]\right]\leq C,$
    \item[$(iii)$] $\mathbb E\left[\sum_{1\leq i<j\leq n}\mathbbm 1 [\n_{i,j}\geq 3]\right]\leq \frac{C}{n^{3/2}}.$
\end{enumerate}
\end{lem}
\begin{rem} The above result implies that with high probability: there are at most $O(\sqrt{n})$ open edges, there is a bounded number of edges of weight $2$, and there are no edges of weight at least $3$.
\end{rem}
\begin{proof} We first show the result for $\mathbb P=\mathbb P^{S}_{n,\lambda}$.
We begin by proving $(i)$. Consider $1\leq i<j\leq n$. Let $\n$ be a current such that $\n_{i,j}\geq 1$, and define the current $\m$ that satisfies $\m_{i,j}=\n_{i,j}-1$ and otherwise coincides with $\n$. Then, $\m$ has sources $S\Delta \{i,j\}$ and $w(\n)\leq d_n(\lambda)w(\m)$. Thus,
\begin{equs}
\mathbb{P}^{S}_{n,\lambda}\left[\n_{i,j}\geq 1\right]\leq d_n(\lambda) \frac{Z^{S\Delta \{i,j\}}_n(\lambda)}{Z^S_n(\lambda)}=(1+o(1))\: d_n(\lambda)\: (c_n n)^{|S|- |S\Delta \{i,j\}|}\frac{\left\langle\varphi^{|S\Delta \{i,j\}|}\right\rangle_0}{\left\langle\varphi^{|S|}\right\rangle_0}.
\end{equs}
Summing over all $i<j$, and splitting the sum according to the three possibilities $|S|-|S\Delta\{i,j\}|\in \{-2,0,2\}$, we see by Lemma~\ref{lem: GS} that 
\begin{multline}\label{eq:ref estimates}
    \mathbb E^S_{n,\lambda}\left[\sum_{1\leq i<j\leq n}\mathbbm 1 [\n_{i,j}\geq 1]\right]\\\leq O\left(n^2 d_n(\lambda) \frac{1}{\sqrt{n}}\right)+O(nd_n(\lambda))+O(d_n(\lambda)\sqrt{n})=O(\sqrt{n}).
\end{multline} 
We now turn to $(ii)$. As above, let $\n$ be a current such that $\n_{i,j}\geq 2$, and define the current $\m$ that satisfies $\m_{i,j}=\n_{i,j}-2$ and otherwise coincides with $\n$. Then, $\m$ has sources $S$ and $w(\n)\leq d_n(\lambda)^2w(\m)$. Thus $\mathbb{P}^{S}_{n,\lambda}\left[\n_{i,j}\geq 2\right]\leq d_n(\lambda)^2$,
which implies that
\begin{equs}
    \mathbb E^S_{n,\lambda}\left[\sum_{1\leq i<j\leq n}\mathbbm 1 [\n_{i,j}\geq 2]\right]=O(1).
\end{equs}  
Finally, we prove $(iii)$. We have
\begin{equs}
\mathbb{P}^{S}_{n,\lambda}\left[ \n_{i,j}\geq 3\right]\leq d_n(\lambda)^3 \frac{Z^{S\Delta \{i,j\}}_n(\lambda)}{Z^S_n(\lambda)}=(1+o(1))d_n(\lambda)^3(c_n n)^{|S|- |S\Delta \{i,j\}|}\frac{\left\langle\varphi^{|S\Delta \{i,j\}|}\right\rangle_0}{\left\langle\varphi^{|S|}\right\rangle_0}.
\end{equs}
Hence, proceeding as above,
\begin{equs}
    \mathbb E^S_{n,\lambda}\left[\sum_{1\leq i<j\leq n}\mathbbm 1 [\n_{i,j}\geq 3]\right]\leq O(n^{-3/2}).
\end{equs} 
The proof for $\mathbb E=\mathbb P^{S,\emptyset}_{n,\lambda}$ is very similar. For $(i)$, it follows from the proof above together with the observation that 
\begin{equs}
    \mathbb P^{S,\emptyset}_{n,\lambda}[\n_{i,j}\geq 1]\leq \mathbb P^S_{n,\lambda}[\n_{i,j}\geq 1]+\mathbb P^\emptyset_{n,\lambda}[\n_{i,j}\geq 1].
\end{equs}
For $(ii)$, we observe that,
\begin{equs}
    \mathbb P^{S,\emptyset}_{n,\lambda}[\n_{i,j}\geq 2]\leq \mathbb P^S_{n,\lambda}[\n_{i,j}\geq 2]+\mathbb P^{\emptyset}_{n,\lambda}[\n_{i,j}\geq 2]+\mathbb P^{S,\emptyset}_{n,\lambda}[\n_1(i,j)\geq 1, \: \n_2(i,j)\geq 1].
\end{equs}
We handle the sum over $1\leq i<j \leq n$ of the first two probabilities using the same method as above. For the last one, notice that
\begin{eqnarray*}
    \mathbb P^{S,\emptyset}_{n,\lambda}[\n_1(i,j)\geq 1, \: \n_2(i,j)\geq 1]&\leq& d_n(\lambda)^2\frac{Z^{S\Delta\{i,j\},\{i,j\}}_n(\lambda)}{Z^{S,\emptyset}_n(\lambda)}
    \\&=&(1+o(1))d_n(\lambda)^2 (c_n n)^{|S|-|S\Delta \{i,j\}|-2}\frac{\langle \varphi^{|S\Delta \{i,j\}|}\rangle_0 \langle \varphi^{2}\rangle_0}{\langle \varphi^{|S|}\rangle_0}.
\end{eqnarray*}
Finally, we proceed similarly for $(iii)$ by noticing that
\begin{multline*}
    \mathbb P^{S,\emptyset}_{n,\lambda}[\n_{i,j}\geq 3]\leq \mathbb P^S_{n,\lambda}[\n_{i,j}\geq 3]+\mathbb P^{\emptyset}_{n,\lambda}[\n_{i,j}\geq 3]+\mathbb P^{S,\emptyset}_{n,\lambda}[\n_1(i,j)\geq 2, \: \n_2(i,j)\geq 1]\\+\mathbb P^{S,\emptyset}_{n,\lambda}[\n_1(i,j)\geq 1, \: \n_2(i,j)\geq 2].
\end{multline*}
\end{proof}
For $x\in V_n$ and a current $\n$, we let $\Delta_\n(x):=\sum_{y\neq x}\n_{x,y}$ be the degree of $x$ in the multigraph associated with $\n$.
\begin{lem}[The degree is either $1,2$ or $4$]\label{lem: degree} There exists $C>0$ such that for all $n$ large enough:
\begin{enumerate}
    \item[$(i)$] $\mathbb E\left[\sum_{1\leq i<j\leq n}\mathbbm 1 [\n_{i,j}\geq 2]\mathbbm 1 [\Delta_\n(i)\geq 3 \textup{ or } \Delta_\n(j)\geq 3]\right]\leq \frac{C}{n^{1/2}},$
    \item[$(ii)$] for all $i\in S$, $\mathbb P\left[\Delta_\n(i)\geq 2\right]\leq \frac{C}{n^{1/2}},$
    \item[$(iii)$] $\mathbb E\left[\sum_{x\in V_n}\mathbbm 1 [\Delta_\n(x)\geq  4]\right]\leq C,$
    \item[$(iv)$] $\mathbb E\left[\sum_{x\in V_n}\mathbbm 1 [\Delta_\n(x)\geq 5]\right]\leq \frac{C}{n^{1/2}}.$
\end{enumerate}
\end{lem}
\begin{rem}
The first property above proves that double edges (i.e. of multiplicity at least two) are isolated with high probability (w.h.p). The second one shows that the sources have degree one w.h.p., the third one states the tightness of the total number of vertices of degree $4$, and finally, the fourth inequality above rules out w.h.p the existence of vertices of degree at least $5$.
\end{rem}
\begin{proof} We prove the result for $\mathbb P=\mathbb P^S_{n,\lambda}$. The proof for $\mathbb P^{S,\emptyset}_{n,\lambda}$ uses similar arguments as in the proof of Lemma \ref{lem: multiplicativity edges}. 

For $(i)$ we already proved in Lemma \ref{lem: multiplicativity edges} that there exists $C_1>0$ such that  
\begin{equs}
    \mathbb E^S_{n,\lambda}\left[\sum_{1\leq i<j\leq n}\mathbbm 1 [\n_{i,j}\geq 3]\right]\leq \frac{C_1}{n^{3/2}}.
\end{equs}
Hence, it remains to handle $\mathbb E^S_{n,\lambda}\left[\sum_{1\leq i<j<k\leq n}\mathbbm 1 [\n_{i,j}= 2,\:\n_{j,k}\geq 1]\right]$.
We have 
\begin{equs}
\mathbb{P}^{S}_{n,\lambda}\left[\n_{i,j}=2,\:\n_{j,k}\geq 1\right]\leq d_n(\lambda)^3 \frac{Z^{S\Delta \{j,k\}}_n(\lambda)}{Z^S_n(\lambda)},
\end{equs}
hence, taking into account whether $j,k$ belong to $S$ or not (as in \eqref{eq:ref estimates}),
\begin{equs}
   \mathbb E^S_{n,\lambda}\left[\sum_{1\leq i<j<k\leq n}\mathbbm 1 [\n_{i,j}= 2,\:\n_{j,k}\geq 1]\right] \leq O(n^{-1/2}).
\end{equs} 
For $(ii)$, we have already showed that
\begin{equs}
\mathbb{P}^{S}_{n,\lambda}\left[\n_{i,j}\geq 2\right]\leq d_n(\lambda)^2, 
\end{equs}
hence, for a fixed $i\in S$,
\begin{equs}
\mathbb{E}^{S}_{n,\lambda}\left[\sum_{j\neq i}\mathbbm 1[\n_{i,j}\geq 2]\right]=O(1/n). 
\end{equs}
It remains to estimate $\mathbb{E}^{S}_{n,\lambda}\left[\sum_{1\leq j<k\leq n}\mathbbm 1[\n_{i,j}\geq 1,\: \n_{i,k}\geq 1]\right]$.
Note that
\begin{equs}
\mathbb{P}^{S}_{n,\lambda}\left[ \n_{i,j}\geq 1, \:\n_{i,k}\geq 1\right]\leq d_n(\lambda)^2 \frac{Z^{S\Delta \{j,k\}}_n(\lambda)}{Z^S_n(\lambda)},
\end{equs}
which yields,
\begin{equs}
\mathbb{E}^{S}_{n,\lambda}\left[\sum_{j<k}\mathbbm 1[\n_{i,j}\geq 1, \:\n_{i,k}\geq 1]\right]\leq O(n^{-1/2}).
\end{equs}
For $(iii)$, we only need to estimate $\mathbb{E}^{S}_{n,\lambda}\left[\sum_{1\leq i<j<k<\ell<m\leq n}\mathbbm 1\left[\n_{i,x}\geq 1, \:\: \forall x\in\{j,k,\ell,m\}\right]\right]$. Indeed, if $i$ has an outgoing double edge, then with high probability it has degree $2$ by $(i)$. To estimate the latter expectation, note that 
\begin{equs}
\mathbb{P}^{S}_{n,\lambda}\left[ \n_{i,x}\geq 1, \: \forall x\in \{j,k,\ell,m\}\right]\leq d_n(\lambda)^4 \frac{Z^{S\Delta \{j,k,\ell,m\}}_n(\lambda)}{Z^S_n(\lambda)},
\end{equs}
hence, letting $r=|\{j,k,\ell,m\}\cap S^c|$ we see that the main contribution comes from when $r=4$, and we get using Lemma \ref{lem: GS} again
\begin{equs}\label{eq: degree 4}
\mathbb{E}^{S}_{n,\lambda}\left[\sum_{1\leq i<j<k<\ell<m\leq n}\mathbbm 1\left[\n_{i,x}\geq 1, \: \forall x\in\{j,k,\ell,m\}\right]\right]\leq  (1+o(1))\tilde{g}^4\frac{\langle \varphi^{|S|+4}\rangle_{0}}{\langle \varphi^{|S|} \rangle_{0}}.   
\end{equs}
Similarly, for $(iv)$, we have 
\begin{equs}
\mathbb{P}^{S}_{n,\lambda}\left[ \n_{i,x}\geq 1,\: \forall x\in \{j,k,\ell,m,y\}\right]\leq d_n(\lambda)^5 \frac{Z^{S\Delta \{i,j,k,\ell,m,y\}}_n(\lambda)}{Z^S_n(\lambda)}
\end{equs}
hence
\begin{equs}
\mathbb{E}^{S}_{n,\lambda}\left[\sum_{1\leq i<j<k<\ell<m<y\leq n}\mathbbm 1\left[\n_{i,x}\geq 1,\: \forall x\in\{j,k,\ell,m,y\}\right]\right]&=O\left(d_n(\lambda)^5\sum_{r=0}^6 n^r (c_n n)^{6-2r}\right)\\ &=O(n^{-1/2}).  
\end{equs}
\end{proof}

The two above results have the following interesting consequence for the number of vertices with non-zero degree. For a current $\n$, let $\mathcal{N}(\n)$ be the number of vertices $i\in V_n$ such that $\Delta_{\n}(i)>0$. Recall that for $g>0$, $\tilde{g}=(12g)^{1/4}$.

\begin{prop}\label{prop: gaussian}
Let $\lambda\in \mathbb R$, $g>0$, and $k\geq 0$. Set $S=\lbrace 1,\ldots, 2k\rbrace$ if $k\geq 1$ and $S=\emptyset$ otherwise. Let $\n$ be distributed according to $\mathbb P_{n,\lambda}^S$. Then,
\begin{equs}
\frac{\mathcal{N}(\n)}{(c_n n)^2}\underset{n \rightarrow\infty}{\longrightarrow} \frac{\tilde{g}^4}{2} Z^2_k,
\end{equs}
where the convergence holds in distribution, and $Z_k$ is the random variable with density\footnote{Recall that $a=a(\lambda)$ was defined in the beginning of Section \ref{section: properties of the current}.} 
\begin{equs}
    \frac{\varphi^{2k}}{\left\langle\varphi^{2k}\right\rangle_0}\mathrm{d} \rho_{g,a(\lambda)}.
\end{equs}
\end{prop}
\begin{proof}
By Lemma \ref{lem: degree}, ``most vertices'' of positive degree have degree $2$. As a result, with high probability we have that $\mathcal{E}(\n):=\sum_{1\leq i<j\leq n}\mathbbm{1}[\n_{i,j}\geq 1]=(1+o(1))\mathcal{N}(\n)$. Hence, it suffices to prove the convergence (in distribution) of $\mathcal{E}(\n)/(c_n n)^2$.

We will first show convergence of all moments. Let $r>0$. Recall that by Lemma \ref{lem: multiplicativity edges}, most open edges are simple (i.e. carry weight $1$) with high probability. As a result, in order to estimate $\mathbb{E}^{S}_{n,\lambda}\left[\frac{\mathcal{E}^r}{(c_n n)^{2r}}\right]$, it is sufficient to consider the probabilities $\mathbb{P}^{S}_{n,\lambda}\left[ \n_{i_s,j_s}= 1,\: \forall s\in \{1,2\ldots,r\}\right]$ for $i_1,j_1, i_2,j_2, \ldots, i_r,j_r\in V_n$ not necessarily distinct. Letting $\ell$ be the number of distinct edges $\{i_s,j_s\}$ and $m=|S\Delta \{i_1,j_1,\ldots,i_r,j_r\}|$ we have
\begin{equs}
    \mathbb{P}^{S}_{n,\lambda}\left[ \n_{i_s,j_s}= 1,\: \forall s\in \{1,2\ldots,r\}\right]&= d_n(\lambda)^{\ell} \frac{Z^{S\Delta \{i_1,j_1,\ldots,i_r,j_r\}}_n(\lambda)[\n_{i_s,j_s}=0,\: \forall s\in \{1,2\ldots,r\}]}{Z^S_n(\lambda)}
    \\
    &=(1+o(1))\: d_n(\lambda)^{\ell} \frac{Z^{S\Delta \{i_1,j_1,\ldots,i_r,j_r\}}_n(\lambda)}{Z^S_n(\lambda)}
    \\
    &=(1+o(1))\: d_n(\lambda)^{\ell} \: (c_n n)^{2k-m}\frac{\left\langle\varphi^m\right\rangle_0}{\left\langle\varphi^{2k}\right\rangle_0},
\end{equs}
where we used that,
\begin{align*}
    \frac{Z^{S\Delta \{i_1,j_1,\ldots,i_r,j_r\}}_n(\lambda)[\n_{i_s,j_s}=0,\: \forall s\in \{1,2\ldots,r\}]}{Z^{S\Delta \{i_1,j_1,\ldots,i_r,j_r\}}_n(\lambda)}&= \mathbb{P}^{S}_{n,\lambda}\left[ \n_{i_s,j_s}= 0,\: \forall s\in \{1,2\ldots,r\}\right]\\&=1-o(1).
\end{align*}
Summing these probabilities we see that the main contribution comes from the case where $i_1,j_1, i_2,j_2, \ldots, i_r,j_r\in V_n\setminus S$ are all distinct.
Thus,
\begin{equs}
\mathbb{E}^{S}_{n,\lambda}\left[\frac{\mathcal{E}^r}{(c_n n)^{2r}}\right]&=(1+o(1)) \frac{n^{2r}}{2^r} d_n(\lambda)^r (c_n n)^{-4r} \frac{\left\langle\varphi^{2k+2r}\right\rangle_0}{\left\langle\varphi^{2k}\right\rangle_0}\\
&= (1+o(1)) \frac{\tilde{g}^{4r}}{2^r} \frac{\left\langle\varphi^{2k+2r}\right\rangle_0}{\left\langle\varphi^{2k}\right\rangle_0}\\
&= (1+o(1)) \mathbb{E}\left[\frac{\tilde{g}^{4r}}{2^r} Z_k^{2r}\right],  
\end{equs}
where $Z_k$ is a random variable with density $\frac{\varphi^{2k} }{\left\langle\varphi^{2k}\right\rangle_0} \mathrm{d} \rho_{g,a}$.

The convergence of all moments implies that the sequence $\left(\mathcal{N}(\n)/(c_n  n)^2\right)_{n\geq 1}$ is tight, and any limiting random variable has $r$-th moment given by $\frac{\tilde{g}^{4r}}{2^r} \: \frac{\left\langle\varphi^{2k+2r}\right\rangle_0}{\left\langle\varphi^{2k}\right\rangle_0}$. Since $\frac{\tilde{g}^4}{2} Z_k^2$ has finite exponential moments, it follows that there exists a unique distribution with $r$-th moment given by $\frac{\tilde{g}^{4r}}{2^r} \: \frac{\left\langle\varphi^{2k+2r}\right\rangle_0}{\left\langle\varphi^{2k}\right\rangle_0}$. Thus $\left(\mathcal{N}(\n)/(c_n  n)^2\right)_{n\geq 1}$ converges to $\frac{\tilde{g}^{4}}{2} Z_k^{2}$, as desired.
\end{proof}

\section{Scaling limit of the cluster size distribution}\label{section: scaling limit of cluster size distribution}

In this section, we will use the results of Section \ref{section: properties of the current} and Appendix \ref{appendix: proof of GS} to prove Theorem \ref{thm: main theorem}. 

We begin with a definition.
\begin{defn} Let $n\geq 1$. Let $S$ be an even subset of $V_n$ and $\mathsf{P}=\lbrace P_1,\ldots, P_r\rbrace\in \mathcal{P}_{\mathrm{even}}(S)$. Let $\mathbf{a}, \mathbf{b}\in (\mathbb R_{\geq 0})^r$ with $a_i\leq b_i$ for each $1\leq i\leq r$. Define the subset $\mathcal{I}_n(\mathbf{a},\mathbf{b},\mathsf{P})$ of $\Omega_n$ as
\begin{equs}
    \mathcal{I}_n(\mathbf{a},\mathbf{b},\mathsf{P}):=\left\{\n\in \Omega_n, \: \mathfrak{P}(\n)=\mathsf{P}, \: a_i\sqrt{n}\leq |\mathcal{C}_{P_i}(\n)|\leq b_i\sqrt{n} \: \forall \: 1\leq i \leq r\right\},
\end{equs}
where for each $1\leq i \leq r$, $\mathcal{C}_{P_i}(\n)$ is the cluster of $\n$ containing the elements of $P_i$.
\end{defn}
As we will see, it is sufficient to study the above events to obtain Theorem \ref{thm: main theorem}. In fact, the latter will be a consequence of the following result.
\begin{prop}\label{prop: main prop} Let $\lambda\in \mathbb R$ and $k\geq 1$. Set $S:=\lbrace 1,\ldots, 2k\rbrace$ and consider $\mathsf{P}=\lbrace P_1,\ldots, P_r\rbrace\in \mathcal{P}_{\mathrm{even}}(S)$. Let $\mathbf{a}, \mathbf{b}\in \mathbb R_{\geq 0}^r$ with $a_i\leq b_i$ for each $1\leq i\leq r$. Then, the following limits exist
\begin{equs}
   \lim_{n\rightarrow \infty}\mathbb P_{n,\lambda}^S[\mathcal{I}_n(\mathbf{a},\mathbf{b},\mathsf{P})], \qquad \lim_{n\rightarrow \infty}\mathbb P_{n,\lambda}^{S,\emptyset}[\mathcal{I}_n(\mathbf{a},\mathbf{b},\mathsf{P})].
\end{equs}
Moreover, both limits are continuous as functions of $\mathbf{a}$ and $\mathbf{b}$.
\end{prop}
We now show how to obtain our main result using Proposition \ref{prop: main prop}.
\begin{proof}[Proof of Theorem~\textup{\ref{thm: main theorem}}] We give the proof for the case of $\mathbb P_{n,\lambda}^S$, with the proof for the double current measure being similar. The main observation we need to make is that with high probability the clusters of the sources always remain of order at most $(c_n n)^2\asymp \sqrt{n}$. This fact follows immediately from the results of Section \ref{section: properties of the current}. Indeed, using Lemma \ref{lem: multiplicativity edges} $(i)$ and Markov's inequality, we get the existence of $C>0$ such that, for any $A>0$,
\begin{equs}
    \limsup_{n\rightarrow \infty}\mathbb P^S_{n,\lambda}\left[\cup_{1\leq i \leq 2k}\lbrace |\mathcal{C}_i(\n)|\geq A\sqrt{n}\right\rbrace]\leq \frac{2k C}{A}.
\end{equs}
In particular, this proves that the sequence of random variables $\left(\frac{|\mathcal{C}_1(\n)|}{\sqrt{n}},\ldots, \frac{|\mathcal{C}_{2k}(\n)|}{\sqrt{n}}, \mathfrak{P}(\n)\right)_{n\geq 1}$ is tight on $ (\mathbb R_{\geq 0})^{2k}\times \mathcal{P}_{\mathrm{even}}(S)$ (equipped with the standard product $\sigma$-algebra). Proposition \ref{prop: main prop} then fully characterises the limit which yields the desired convergence. The fact that the first $2k$ coordinates of the limiting random variables $X_{k,\lambda}^{\mathsf{s}}$ and $X_{k,\lambda}^{\mathsf{d}}$ are strictly positive almost surely  is a consequence of the second part of Proposition \ref{prop: main prop}, since the event $\mathcal{I}_n(\mathbf{a},\mathbf{b},\mathsf{P})$ is empty if one of the $b_i$'s is $0$.
\end{proof}
It remains to prove Proposition \ref{prop: main prop}. Below, we fix $g>0$, $\lambda\in \mathbb R$ and $k\geq 1$. We recall that $S=\lbrace 1,\ldots, 2k\rbrace$. We work on the complete graph $K_n$ with inverse temperature $d_n(\lambda)$ and study the measures $\mathbb P_{n,\lambda}^{S}$ and $\mathbb P_{n,\lambda}^{S,\emptyset}$.

Let $\mathsf{P}=\lbrace P_1,\ldots, P_r\rbrace\in \mathcal{P}_{\mathrm{even}}(S)$. Let $\mathbf{a}, \mathbf{b}\in \mathbb R_{\geq 0}^r$ with $a_i\leq b_i$ for each $1\leq i\leq r$. We begin by analysing the typical form of a current $\n$ which belongs to $\mathcal{I}_n(\mathbf{a},\mathbf{b},\mathsf{P})$. The results of Section \ref{section: properties of the current} motivate the introduction of the following event.
\begin{defn}[Good event] For $A>0$, let $\mathcal{G}_{A}$ be the event that the current $\n$ satisfies the following properties: for every $1\leq i \leq 2k$ the cluster $\mathcal{C}_i(\n)$ has only simple edges (i.e. edges $e$ with $\n_e=1$) and $\Delta_\n(i)=1$, there are no vertices of degree $5$ or larger, and the number of vertices of degree $4$ is at most $A$.
\end{defn}
Recall that the elements of $S^c$ always have even degree. This means that if $\n\in \mathcal{G}_A$, vertices in $S^c$ have degree $0$, $2$ or $4$ in $\n$. Lemmas \ref{lem: multiplicativity edges} and \ref{lem: degree} imply that there exists $C>0$ such that for all $n\geq 1$,
\begin{equs}\label{eq: good event is really good}
    \mathbb P^S_{n,\lambda}[\mathcal{G}_{A}]\geq 1-\frac{C}{A}-\frac{C}{\sqrt{n}}, \qquad 
    \mathbb P^{S,\emptyset}_{n,\lambda}[\mathcal{G}_{A}]\geq 1-\frac{C}{A}-\frac{C}{\sqrt{n}}.
\end{equs}
As a result, we can restrict our attention to currents $\n$ satisfying $\n\in \mathcal{I}_n(\mathbf{a},\mathbf{b},\mathsf{P})\cap \mathcal{G}_A$. Let $\n$ be such a current. Denote by $B(\n)$ the set of vertices of degree $4$ in $\cup_{i=1}^r \mathcal{C}_{P_i}(\n)$. Then, each cluster $\mathcal{C}_{P_i}(\n)$ consists of a union of self-avoiding walks whose internal (meaning not the endpoints) vertices have degree 2 in $\n$, together with a collection of cycles of length at least $3$ which contain exactly one element of $B(\n)\cap \mathcal{C}_{P_i}(\n)$. Note that these self-avoiding walks begin and end at either an element of $B(\n)$ or at a source. See Figure \ref{fig:backbonegraph}.

This observation motivates the following definition. 
\begin{defn}[Backbone multigraph]\label{def: backbone multigraph} Let $\n \in \mathcal{G}_A$ satisfy $\mathfrak{P}(\n)=\mathsf{P}$. Let $G=(V,E)$ be a multigraph (i.e.\ we allow multiple edges with the same endpoints and loops attached to a vertex) such that $V=\{1,2,\ldots,v\}$ for some $2k\leq v\leq n$, all vertices in $\{1,2,\ldots,2k\}$ have degree $1$, and all vertices in $\{2k+1,\ldots,v\}$ have degree $4$, where we use the standard convention that loops contribute twice to the degree of a vertex. We say that $G$ is a \emph{backbone multigraph} of $\n$ if there exists an automorphism $\pi$ of $K_n$ which fixes each element of $S$ such that: $\pi(V)$ consists of the union of the source set $S$ with the set of vertices of degree $4$ lying in one of the clusters $\mathcal{C}_{P_i}(\n)$; the set $\pi(E)$ is obtained by contracting each self-avoiding walk or cycle in the above description  of $\cup_{i=1}^r\mathcal{C}_{P_i}(\n)$.

We call $G_i$ the connected component of $G$ which contains $P_i$, and we write $\mathcal{B}(G)$ for the event that $G$ is a backbone multigraph of $\n$. Moreover, we define $\Gamma(k,A,\mathsf{P})$ as the set of isomorphism classes of backbone multigraphs for the currents $\n\in \mathcal{G}_A$ such that $\mathfrak{P}(\n)=\mathsf{P}$, where we identify two backbone multigraphs $G$ and $G'$ if there is an isomorphism from $G$ to $G'$ which fixes each element of $S$.
\end{defn}

\begin{figure}[H]
    \centering
    \includegraphics{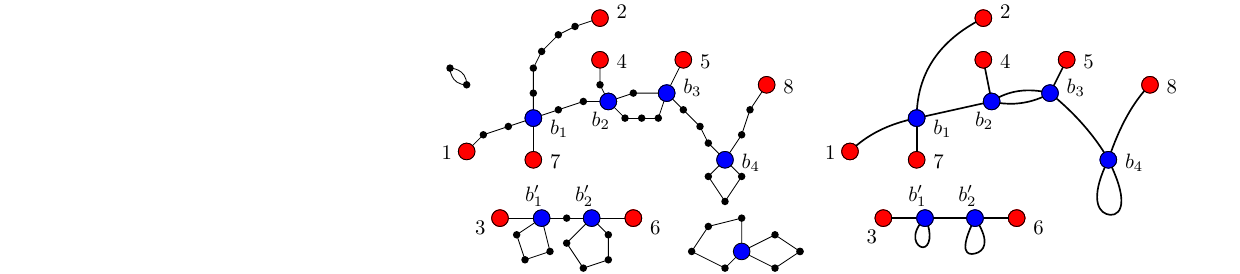}
    \caption{\textsc{Left}: the multigraph associated with a typical sample $\n$ of $\mathbb P^{S}_{n,\lambda}$ with $S=\lbrace 1,\ldots,8\rbrace$. \textsc{Right}: the backbone multigraph associated to $\n$. The sources are represented in red, the set $B(\n)$ of vertices of degree $4$ is represented in blue, and the vertices of degree $2$ in black. The edges of multiplicity $2$ do not appear in the backbone graph since they do not contribute to the cluster of the sources with high probability.}
    \label{fig:backbonegraph}
\end{figure}

We are now equipped to prove Proposition \ref{prop: main prop}. We will use the above observations to show that, under the measures of interest, the probabilities of events of the type $\mathcal{I}_n(\mathbf{a},\mathbf{b},\mathsf{P})\cap \mathcal{G}_A \cap \mathcal{B}(G)$ for a fixed \emph{admissible} (i.e.\ compatible with the event $\mathcal{I}_n(\mathbf{a},\mathbf{b},\mathsf{P})\cap \mathcal{G}_A$) backbone multigraph $G$ converge as $n\rightarrow \infty$. Proposition \ref{prop: main prop} will then follow from summing the above quantities among the (finite) set of admissible backbone multigraphs.

\begin{proof}[Proof of Proposition~\textup{\ref{prop: main prop}}]
We start with the simpler case of a single random current. As we will see, the case of double random currents can be handled similarly, with only a few modifications. We work on the complete graph $K_n$ and consider the random current measure $\mathbb{P}^S_{n,\lambda}$ with inverse temperature $d_n(\lambda)$. Recall that $S:=\{1,\dots, 2k\}$ and that we fixed $\mathsf{P}=\lbrace P_1,\ldots,P_r\rbrace\in \mathcal{P}_{\mathrm{even}}(S)$ and $\mathbf{a}\leq \mathbf{b}\in (\mathbb R_{\geq 0})^r$. 

With the above definitions in hands we have,
\begin{equs}
\mathbb P^S_{n,\lambda}[\mathcal{I}_n(\mathbf{a},\mathbf{b},\mathsf{P}), \: \mathcal{G}_{A}]=\sum_{[G]\in \Gamma(k,A,\mathsf{P})}\mathbb P^S_{n,\lambda}[\mathcal{I}_n(\mathbf{a},\mathbf{b},\mathsf{P}), \: \mathcal{G}_{A}, \: \mathcal{B}(G)],  
\end{equs}
where $G$ is an arbitrary representative of the corresponding isomorphism class.

It remains to estimate the probability $\mathbb P^S_{n,\lambda}[\mathcal{I}_n(\mathbf{a},\mathbf{b},\mathsf{P}), \: \mathcal{G}_{A}, \: \mathcal{B}(G)]$. To this end, let us write $m_i$ for the number of vertices of degree $2$ in $\mathcal{C}_{P_i}(\n)$, and $m$ for the number of vertices of degree $2$ in $\cup_{i=1}^r\mathcal{C}_{P_i}(\n)$. Note that $m=\sum_{i=1}^rm_i$. We also let $\ell_i$ be the number of edges of $G_i$, so that $\ell:=\sum_{i=1}^r \ell_i$ is the number of edges in $G$. Define, $\mathcal{H}_n(\mathbf{a},\mathbf{b},A, G, (m_i)_{1\leq i \leq r})$ to be the set of currents $\n$ which satisfy: $\n \in \mathcal{I}_n(\mathbf{a},\mathbf{b},\mathsf{P})\cap \mathcal{G}_{A}\cap \mathcal{B}(G)$, and $\mathcal{C}_{P_i}(\n)$ contains $m_i$ vertices of degree $2$.

Let $\mathcal{C}_S(\n)=\cup_{i=1}^{2k}\mathcal{C}_i(\n)$. A current $\n\in \mathcal{H}_n(\mathbf{a},\mathbf{b},A, G, (m_i)_{1\leq i \leq r})$ can be written as:
\begin{equs}
    \n=\m_1(\n)+\m_2(\n),
\end{equs}
where $\m_1=\m_1(\n)$ has source set $S$ and is supported on $\mathcal{P}_2(\mathcal{C}_S(\n))$, and $\m_2=\m_2(\n)$ is sourceless and is supported on $\mathcal{P}_2(\mathcal{C}_S(\n)^c)$. Moreover, one has that $w(\n)=w(\m_1)w(\m_2)$ with $w(\m_1)=d_n(\lambda)^{m+\ell}$. Indeed, by definition $\m_1(\n)$ only contains edges of multiplicity one. Hence, its weight is exactly $d_n(\lambda)^{|\lbrace e, \: \m_1(e)>0\rbrace|}$, and $|\lbrace e, \: \m_1(e)>0\rbrace|$ is nothing but the number of open edges in $\mathcal{C}_S(\n)$. By the handshake lemma applied to $G$ we get that $2\ell=2k+4(v-2k)$. Now, applying the handshake lemma to
$\mathcal{C}_{S}(\n)$ we get that its number of edges is equal to $k+m+2(v-2k)=m+\ell$.\footnote{An alternative way to see this is to observe that every self-avoiding walk and every cycle that corresponds to an edge or a loop of $G$ contains one more edge than vertices of degree $2$.} Hence, for each possible value $\m$ that $\m_1(\n)$ may take,
\begin{equs}\label{eq: proof thm 1}
    \sum_{\n \in \mathcal{H}_n(\mathbf{a},\mathbf{b},A, G, (m_i)_{1\leq i \leq r}), \:\m_1(\n)=\m}w(\n)=d_n(\lambda)^{m+\ell}\times \sum_{\m_2\in \Omega_{n-v-m},  \: \sm_2=\emptyset}w_{n,\lambda}(\m_2).
\end{equs}
Let $x=m/\sqrt{n}$. Notice that, 
\begin{equs}
    d_n(\lambda)=\frac{1}{n}\left(1-\frac{\lambda}{\sqrt{n}}\right)=\frac{1}{n-v-m}\left(1-(1+o(1))\frac{\lambda+x}{\sqrt{n-v-m}}\right), 
\end{equs}
where the $o(1)$ term goes to $0$ uniformly over $x$ in a bounded interval.
Thus, one obtains that
\begin{equs}\label{eq: proof thm 2}
    \sum_{\m_2\in \Omega_{n-v-m}, \: \sm_2=\emptyset}w_{n,\lambda}(\m_2)=Z^\emptyset_{n-v-m}\big[(\lambda+x)(1+o(1))\big].
\end{equs}
Since the right-hand side of \eqref{eq: proof thm 1} does not depend on the value of $\m_1(\n)$, it remains to count the number of values $\m_1(\n)$ may take for $\n \in \mathcal{H}_n(\mathbf{a},\mathbf{b},A, G)$. We claim that the number of possibilities for $\m_1(\n)$ is equal to 
\begin{multline}\label{eq: number}
\frac{1+o(1)}{2^L |\mathfrak{A}_G|}\sum_{m_1=a_1 \sqrt{n}}^{b_1\sqrt{n}}\cdots\sum_{m_r=a_r \sqrt{n}}^{b_r \sqrt{n}} {n - 2k \choose v - 2k} (v-2k)! {n-v\choose m} \\ \hspace{6cm} \times  {m \choose m_1,\ldots, m_r}\prod_{i=1}^r m_i!\frac{m_i^{\ell_i-1}}{(\ell_i-1)!}=
\\ \frac{1+o(1)}{2^L |\mathfrak{A}_G|}\sum_{m_1=a_1 \sqrt{n}}^{b_1\sqrt{n}}\cdots\sum_{m_r=a_r \sqrt{n}}^{b_r \sqrt{n}} {n - 2k \choose v - 2k} (v-2k)! {n-v\choose m} m!\prod_{i=1}^r\frac{m_i^{\ell_i-1}}{(\ell_i-1)!},
\end{multline}
where $L$ is the number of loops of $G$, and $\mathfrak{A}_G$ is the set of automorphisms of $G$ which fix each element of $S$. To see this, note that we have ${n - 2k \choose v - 2k}$ ways to choose the vertices of degree $4$ in $\mathcal{C}_{S}(\n)$, and $\frac{(v-2k)!}{|\mathfrak{A}_G|}$ ways to associate each of these vertices to a vertex of $G$ in a unique way, see Figure~\ref{fig:isombackbonegraph}. 
\begin{figure}[H]
    \centering
    \includegraphics{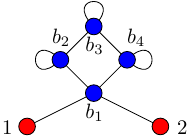}
    \caption{A backbone multigraph $G$ with two automorphisms in $\mathfrak{A}_G$.}
    \label{fig:isombackbonegraph}
\end{figure}
Moreover, there are ${n-v\choose m}$ ways to choose the $m$ vertices of degree $2$. Next, we split these $m$ vertices into ordered
sets of size $m_1,m_2,\ldots,m_r$, corresponding to the vertices of degree $2$ in each $\mathcal{C}_{P_i}(\n)$, for which there are ${m \choose m_1,\ldots, m_r}$ possibilities. Then, we split each of these $m_i$ vertices into
$\ell_i$ sets corresponding to the $\ell_i$ edges of $G_i$, and order them within the sets. These ordered sets may be empty except for those corresponding to the cycles of $ \mathcal{C}_S(\n)$ that are attached to a vertex of degree $4$. Thus, the number of possibilities to split them into $\ell_i$ ordered sets is sandwiched between $m_i!{m_i+\ell_i-1 \choose \ell_i-1}-L m_i! \binom{m_i+\ell_i-2}{\ell_i-2}=\left(1+o(1)\right)m_i!\frac{m_i^{\ell_i-1}}{(\ell_i-1)!}$\footnote{Here the term $m_i! \binom{m_i+\ell_i-2}{\ell_i-2}$ upper bounds the number of possibilities in which a particular set corresponding to one of the loops is empty, and we recall that $L$ is the number of loops.} and $m_i!{m_i+\ell_i-1 \choose \ell_i-1}=\left(1+o(1)\right)m_i!\frac{m_i^{\ell_i-1}}{(\ell_i-1)!}$. Finally, the term $2^{-L}$ accounts for the cyclic ordering of the vertices in the cycles corresponding to the loops of $G$.

Gathering \eqref{eq: proof thm 1}, \eqref{eq: proof thm 2}, and \eqref{eq: number} yields that
$\mathbb P^S_{n,\lambda}[\mathcal{I}_n(\mathbf{a},\mathbf{b},\mathsf{P}), \: \mathcal{G}_{A}, \: \mathcal{B}(G)]$ is equal to
\begin{multline}\label{eq: proof thm 3}
\frac{1+o(1)}{Z^{S}_n(\lambda) 2^L |\mathfrak{A}_G|}\sum_{m_1=a_1 \sqrt{n}}^{b_1\sqrt{n}}\cdots\sum_{m_r=a_r \sqrt{n}}^{b_r \sqrt{n}} {n - 2k \choose v - 2k}(v-2k)! \\ \times{n-v\choose m} m! \:\prod_{i=1}^r\frac{m_i^{\ell_i-1}}{(\ell_i-1)!}\:d_n(\lambda)^{m+\ell} Z^\emptyset_{n-v-m}\big[(\lambda+x)(1+o(1))\big].
\end{multline}
We now proceed with the asymptotic analysis of each of these terms. Note that for $x_i=m_i/\sqrt{n}$ we have
\begin{equs}
    {n - 2k \choose v - 2k}(v-2&k)!=(1+o(1)) n^{v-2k}, \quad   
    d_n(\lambda)^{m+\ell}=(1+o(1))\: n^{-m-\ell}  e^{-\lambda x}, \\
    & \qquad \qquad m_i^{\ell_i-1}=\: x_i^{\ell_i-1} n^{(\ell_i-1)/2}.
\end{equs}
Then, by the asymptotic estimate of the partition function obtained in Proposition \ref{prop: estimating the partition function},
\begin{equs}
 Z^\emptyset_{n-v-m}\big[(\lambda+x)(1+o(1))\big]=(1+o(1)) \frac{n^{1/4}}{\sqrt{2\pi e}} \int_{\mathbb{R}}  e^{-\frac{1}{12}s^4-\frac{\lambda+x}{2}s^2} \mathrm{d}s, 
\end{equs}
where the $o(1)$ term on the right-hand side is uniform in $x$ thanks to the monotonicity of $Z^\emptyset_{k}(u)$ in $u$.
Moreover, we have
\[
\log{\left\{\frac{(n-v)!}{(n-v-m)! n^m}\right\}} = \sum_{i=v}^{v+m-1} \log{\left(\frac{n-i}{n}\right)} = \sum_{i=v}^{v+m-1}\left(-\frac{i}{n} + O\left(\frac{i^2}{n^2}\right) \right)= -\frac{m^2}{2n} + o(1),
\]
hence,
\begin{equs}
{n-v\choose m} \: m!=(1+o(1)) \: e^{-x^2/2} n^m.
\end{equs}
Multiplying all these together, and recalling that $2\ell=4(v-2k)+2k$, we see that the exponent of $n$ is equal to
\begin{equs}
v - 2k +(\ell-r)/2-\ell +1/4=v-2k-\ell/2-(r/2-1/4)=-k/2-(r/2-1/4).
\end{equs}
Thus, using Riemann approximation, we obtain that \eqref{eq: proof thm 3} is equal to
\begin{equs}
(1+o(1))\frac{n^{-k/2+1/4}}{\sqrt{2\pi e}  2^L |\mathfrak{A}_G|}\frac{1}{n^{r/2}}\sum_{m_1=a_1 \sqrt{n}}^{b_1\sqrt{n}}\cdots\sum_{m_r=a_r \sqrt{n}}^{b_r \sqrt{n}} e^{-x^2/2-\lambda x}\prod_{i=1}^r \frac{x_i^{\ell_i-1}}{(\ell_i-1)!}  \int_{\mathbb{R}} e^{-\frac{1}{12}s^4-\frac{\lambda+x}{2}s^2}\mathrm{d}s \\
=(1+o(1))  \frac{n^{-k/2+1/4}}{\sqrt{2\pi e}  2^L |\mathfrak{A}_G|}\int_{a_1 }^{b_1} \ldots \int_{a_r}^{b_r}\int_{\mathbb{R}}  \prod_{i=1}^r \frac{x_i^{\ell_i-1}}{(\ell_i-1)!}\: e^{-\frac{1}{2}x^2-\lambda x-\frac{1}{12}s^4-\frac{\lambda+x}{2}s^2}\prod_{i=1}^r\mathrm{d} x_i \mathrm{d}s.
\end{equs}
By Proposition~\ref{prop: estimating the partition function},  $Z^S_n(\lambda)=(1+o(1))\: K_S(\lambda) \: n^{-k/2+1/4}$ for some constant $K_S(\lambda)>0$, hence 
$\mathbb P^S_{n,\lambda}[\mathcal{I}_n(\mathbf{a},\mathbf{b},\mathsf{P}), \: \mathcal{G}_{A}, \: \mathcal{B}(G)]$ is asymptotically equal to
\begin{equs}\label{eq: eq limit single}
\frac{1+o(1)}{ K_S(\lambda) \sqrt{2\pi e}  2^L \mathfrak{A}_G}\int_{a_1 }^{b_1} \ldots \int_{a_r}^{b_r}\int_{\mathbb{R}} \prod_{i=1}^r \frac{x_i^{\ell_i-1}}{(\ell_i-1)!}\:  e^{-\frac{1}{2}x^2-\lambda x-\frac{1}{12}s^4-\frac{\lambda+x}{2}s^2}\prod_{i=1}^r \mathrm{d} x_i \mathrm{d}s.
\end{equs}

This completes the proof in the case of a single random current, and now we proceed with the proof in the case of the measure $\mathbb P^{S,\emptyset}_{n,\lambda}$. Introduce the event $\mathcal{F}_n$ that $\Delta_{\n}(i)=0$ for every $i\in S$, and recall that Lemma \ref{lem: degree} $(ii)$ implies that $\mathbb{P}^{\emptyset}_{n,\lambda}[\mathcal{F}_n]=1-O(1/\sqrt{n})$. Thus, it suffices to estimate the probability $\mathbb P^{S,\emptyset}_{n,\lambda}[\n_1+\n_2\in \mathcal{I}_n(\mathbf{a},\mathbf{b},\mathsf{P}), \: \n_1+\n_2\in\mathcal{G}_{A}, \: \n_2\in \mathcal{F}_n, \: \mathcal{B}(G)]$. Using the same notations as above, we see that the weight of $\m_1(\n_1+\n_2)$ is again equal to $d_n(\lambda)^{m+\ell}$, and the main differences in the estimates come from the number of possibilities for $\m_1(\n_1+\n_2)$.

Given a realisation of the clusters $\mathcal{C}_{P_i}(\n_1+\n_2)$ there are several ways to assign each open edge of these clusters to either $\n_1$ or $\n_2$. Observe that if $u\in \mathcal{C}_{P_i}(\n_1+\n_2)$ is a vertex of degree $2$ with neighbours $z$ and $w$, then either $\n_1(u,z)=\n_1(u,w)=1$ or $\n_2(u,z)=\n_2(u,w)=1$. As a consequence, $\mathcal{C}_{P_i}(\n_1+\n_2)$ consists of self-avoiding walks and cycles which are open either in $\n_1$ or in $\n_2$. This implies that the number of ways to assign each open edge of $\mathcal{C}_{S}(\n_1+\n_2)$ to either $\n_1$ or $\n_2$ is equal to number of ways to colour the edges $G$ red and blue in such a way that the edges incident to $S$ are all coloured red, and the graph consisting of red edges has even degree for every $u\in V\setminus S$, where we recall that loops contribute twice. We denote this number $C_G$. We can now conclude that the number of possibilities for $\m_1(\n)$ is equal to
\begin{equs}
\frac{(1+o(1))C_G}{2^L |\mathfrak{A}_G|}\sum_{m_1=a_1\sqrt{n}}^{b_1\sqrt{n}}\cdots\sum_{m_r=a_r \sqrt{n}}^{b_r\sqrt{n}} {n - 2k \choose v - 2k} (v-2k)! {n-v\choose m} m!\prod_{i=1}^r\frac{m_i^{\ell_i-1}}{(\ell_i-1)!}.
\end{equs}
Using once again Proposition \ref{prop: estimating the partition function},
\begin{equs}
Z^{\emptyset,\emptyset}_{n-v-m}[(\lambda+x)(1+o(1)]=(1+o(1)) \: \frac{n^{1/2}}{2\pi e} \left(\int_{\mathbb{R}}  e^{-\frac{1}{12}s^4-\frac{\lambda+x}{2}s^2}\: \mathrm{d}s\right)^2. 
\end{equs}
Arguing as above, we find that $\mathbb P^{S,\emptyset}_{n,\lambda}[\n_1+\n_2\in\mathcal{I}_n(\mathbf{a},\mathbf{b},\mathsf{P}), \: \n_1+\n_2\in \mathcal{G}_{A}, \: \n_2\in \mathcal{F}_n, \: \mathcal{B}(G)]$
is equal to 
\begin{equs}\label{eq: eq limit double}
 \frac{(1+o(1)) \:C_G 2^{-L}}{K'_S(\lambda)\sqrt{2\pi e} |\mathfrak{A}_G|}\int_{a_1}^{b_1} \ldots \int_{a_r}^{b_r} \left(\int_{\mathbb{R}}  e^{-\frac{1}{12}s^4-\frac{\lambda+x}{2}s^2}\: \mathrm{d}s\right)^2 \prod_{i=1}^r \frac{x_i^{\ell_i-1}}{(\ell_i-1)!} e^{-\frac{1}{2}x^2-\lambda x}\prod_{i=1}^r \mathrm{d} x_i\,  
\end{equs}
for some constant $K'_S(\lambda)>0$.
\end{proof}
We now turn to the proofs of Corollaries \ref{cor: cor 1} and \ref{cor: cor 2}.
\begin{proof}[Proof of Corollary \textup{\ref{cor: cor 1}}] The convergence in distribution is a direct consequence of Theorem \ref{thm: main theorem}. The existence of Gaussian tails for the limiting measures follows from Proposition \ref{prop: gaussian}. Indeed, looking at the case of a single current (the double current case is similar), one has
\begin{equation}
    \Big|\left(\frac{|\mathcal{C}_1(\n_1)|}{\sqrt{n}},\ldots, \frac{|\mathcal{C}_{2k}(\n_1)|}{\sqrt{n}}\right)\Big|_\infty\leq \frac{\mathcal{N}(\n_1)}{\sqrt{n}},
\end{equation}
where $\mathcal{N}(\n_1)$ was defined in Proposition \ref{prop: gaussian} and where $|\cdot|_\infty$ is the $\ell^\infty$ norm on $\mathbb R^d$. Now, by Proposition \ref{prop: gaussian}, the sequence of random variables $(\tfrac{\mathcal{N}(\n_1)}{\sqrt{n}})_{n\geq 1}$ converges in law to a multiple of $Z_k^2$, where $Z_k$ has quartic tails, and hence has Gaussian tails.
\end{proof}

\begin{proof}[Proof of Corollary \textup{\ref{cor: cor 2}}] Again, Theorem \ref{thm: main theorem} gives convergence in law of $(\mathfrak{P}(\n))_{n\geq 1}$. The second part of the statement follows from the explicit formula for the limit obtained in the proof of Proposition \ref{prop: main prop}.
\end{proof}

\begin{rem}[Study of the clusters of points which are not sources]\label{rem: study of the clusters of non-sources} The above analysis does not exclude the possibility of existence of clusters of size of order $\sqrt{n}$ which contain no sources. It is possible to study these clusters via the following observation. Notice that removing an edge of weight $1$ creates at most two new sources. More precisely, if $i,j\in V_n$, we can asymptotically relate the measure $\mathbb P^S_{n,\lambda}[\cdot \: |\: \n_{i,j}=1]$ to the measure $\mathbb P_{n,\lambda}^{S\Delta\lbrace i,j\rbrace}$. This allows to study the scaling limit of the cluster size distribution of $i$ (or $j$) under the measure $\mathbb P^S_{n,\lambda}[\cdot \: |\: \n_{i,j}=1]$.
\end{rem}

\section{Consequences for the $\varphi^4$ model}\label{section: phi4}

The goal of this section is to give a new perspective on the switching lemma for the $\varphi^4$ model obtained in \cite{gunaratnam2022random}, and recalled below. We begin with some useful definitions.

Let $\Lambda=(V,E)$ be a finite graph. Let $g>0$ and $a\in \mathbb R$. The ferromagnetic $\varphi^4$ model on $\Lambda$ with parameters $(g,a)$ at inverse temperature $\beta> 0$ is defined by the finite volume Gibbs equilibrium state: for $F:\mathbb R^V\rightarrow \mathbb R$,
\begin{equs}
    \langle F(\varphi)\rangle_{\Lambda,g,a,\beta}=\frac{1}{Z_{\Lambda,g,a,\beta}}\int F(\varphi)\exp\left(-\beta H_{\Lambda}(\varphi)\right)\prod_{x\in V}\textup{d}\rho_{g,a}(\varphi_x),
\end{equs}
where $Z_{\Lambda,g,a,\beta}$ is the partition function and
\begin{equs}
    H_{\Lambda}(\varphi):=-\sum_{\lbrace x,y\rbrace \in E}\varphi_x\varphi_y.
\end{equs}
Choosing $a>0$, it is clear that the limit of $\rho_{g,a}$ as $g$ goes to $0$ is the Gaussian law $\rho_a$ given by 
\begin{equs}
    \mathrm{d}\rho_{a}(t)
    =
    \frac{1}{z_{a}}e^{-at^2}\mathrm{d}t.
\end{equs}
We will denote by $\langle \cdot \rangle_{0,a}$ the average with respect to $\rho_{a}$.
Hence, if we additionally require that $a>4d\beta$, the corresponding $\varphi^4$ measure on $\Lambda$ converges (as $g$ goes to $0$) to the rescaled massive discrete Gaussian free field (GFF) of mass $m=a-4d\beta$,
defined via the expectation values
\begin{equs}\label{eq: def}
    \langle F\rangle_{\Lambda,a,\beta}^{\mathsf{GFF}}
    :=
    \frac{1}{Z_{\Lambda,a,\beta}^{\mathsf{GFF}}}\int_{\varphi\in \RR^V}F(\varphi)\exp\left(-\beta H_{\Lambda}(\varphi)\right) \prod_{x\in V}\textup{d}\rho_{a}(\varphi_x),
\end{equs}
where $Z_{\Lambda,a,\beta}^{\mathsf{GFF}}$ is the partition function.

Moreover, choosing $a=-2g$ and taking the limit as $g\rightarrow \infty$,
\begin{equs}
    \frac{\delta_{-1}+\delta_1}{2}=\lim_{g\rightarrow \infty}\frac{1}{z_{g,-2g}}e^{-g(\varphi^2-1)^2+g}\textup{d}\varphi,
\end{equs}
the corresponding $\varphi^4$ measure on $\Lambda$ converges to the Ising model on $\Lambda$ defined via the expectation values
\begin{equs}\label{eq: def 2}
    \langle F\rangle_{\Lambda,\beta}^{\mathsf{Ising}}
    :=
    \frac{1}{Z_{\Lambda,\beta}^{\mathsf{Ising}}}\sum_{\sigma\in \lbrace \pm 1\rbrace^V}F(\sigma)\exp\left(-\beta H_{\Lambda}(\sigma)\right).
\end{equs}
The above observations suggest that the $\varphi^4$ model interpolates between the discrete GFF and the Ising model. Interestingly, Theorem \ref{thm: main theorem} allows to extend this observation at the level of the switching lemma for $\varphi^4$: under the appropriate limits mentioned above, one recovers Wick's law (on the GFF side), or the switching lemma (on the Ising model side). More precisely, the proof will reduce to showing that the random variable $P^{\mathsf{d}}_{k,\lambda}$ converges in distribution to either: the uniform measure over pairings in the GFF limit, or the Dirac measure $\delta_{\lbrace 1,\ldots,2k\rbrace}$ in the Ising limit.

This gives an interpretation\footnote{This property can already be seen at the level of the formulas obtained for the tangling probabilities --- see \eqref{eq: eq limit single} and \eqref{eq: eq limit double}, which are expressed as a sum over backbone multigraphs, similarly to Gaussian correlations being expressed as a sum over pairings.} of the switching lemma for $\varphi^4$ as a type of Wick's law for the non-Gaussian $\varphi^4$ field.

\subsection{Switching lemma for $\varphi^4$ and full description of the tangling probabilities}
We recall the switching lemma for $\varphi^4$ in its simplest case and import some terminology from \cite{gunaratnam2022random}. 

Let $\Lambda=(V,E)$ be a finite graph. Write $\Omega_{\Lambda}:= \NN^{E}$ for the space of currents on $\Lambda$. Given $\n \in \Omega_{\Lambda}$, let $\Delta_\n(x):=\sum_{y\in V} \n_{x,y}$ be the $\n$-degree of $x$. We define
\begin{equs}
\cM(\Lambda)
:=
\left\lbrace A \in \mathbb N^V, \: \sum_{x\in V}A_x \text{ is even}\right\rbrace
\end{equs}
to be the set of admissible moments on $\Lambda$, and write $\partial A:=\{x \in V : A_x \text{ is odd}\}$. We will write $A=\emptyset$ if $A_x=0$ for all $x\in \Lambda$. Given $A,B\in \cM(\Lambda)$, define $A+B \in \cM(\Lambda)$ by
\begin{equs}
(A+B)_x
:=
A_x+B_x
\end{equs}
for all $x \in V$.

Fix $A,B\in \cM(\Lambda)$. Let $\n_1,\n_2\in \Omega_{\Lambda}$ satisfy $\partial \n_1 = \partial A$ and $\partial \n_2 = \partial B$. For each $x\in V$, we define the block $\cB_x(\n_1,A)$ as follows: for each $y\in V$, it contains $\n_{x,y}$ points labelled $(xy(k))_{1\leq k \leq \n_1(x,y)}$, and $A_x$ points labelled $(ya(k))_{1\leq k \leq A_x}$. We also define $\cB_x(\n_2,B)$ similarly.
To define a notion of tangled currents for $\n=\n_1+\n_2$, we define $\cB_x(\n,A,B)$ in the following way: for every $y\in \Lambda$, it contains $\n_{x,y}$ points labelled $(xy(k))_{1\leq k \leq \n_{x,y}}$, $A_x$ points labelled $(xa(k))_{1\leq k \leq A_x}$, and $B_x$ points labelled $(xb(k))_{1\leq k \leq B_x}$. There is a natural injection of $\cB_x(\n_1,A)$ and $\cB_x(\n_2,B)$ in $\cB_x(\n,A,B)$. We write $\cT_{\n, A,B}(x)$ for the set of even partitions of the set $\cB_x(\n,A,B)$ whose restriction to $\cB_x(\n_1,A)$ and $\cB_x(\n_2,B)$ is also an even partition, i.e.\ for every $\mathsf{P}=\{P_1,P_2,\ldots, P_r\}\in \cT_{\n, A,B}(x)$, both $|P_i\cap \cB_x(\n_1,A)|$ and $|P_i \cap \cB_x(\n_2,B)|$ are even numbers. The elements of $\cT_{\n, A,B}(x)$ are called {\it admissible} partitions.
We also define
\begin{equs}
\cT_{\n, A,B}
=
\underset{x \in V}{\bigotimes} \,\cT_{\n,A,B}(x)
\end{equs}
to be the set of admissible {\it tanglings}.
For any admissible $A$ and $B$ and any current $\n=\n_1+\n_2$ with $\sn_1=\partial A$ and $\sn_2=\partial B$, we let $\cH(\n,A,B)$
be the graph with vertex set 
\begin{equs}
\bigcup_{z \in V} \cB_x(\n,A,B),
\end{equs}
and edge set 
\begin{equs}
\bigcup_{\lbrace x,y\rbrace\in E}\left\lbrace \lbrace xy(k),yx(k) \rbrace, \: 1\leq k \leq \n_{x,y}\right\rbrace.	
\end{equs}
Given any tangling $\ct \in \cT_{\n,A,B}$, the graph $\cH(\n,A,B)$ naturally projects to a multigraph $\cH(\n,\ct,A,B)$ where for each block $\cB_x(\n,A,B)$, vertices in the same partition class given by $\ct_x$ are identified as one. See Figure~\ref{fig:tanglings}.

\begin{figure}
    \centering
    \includegraphics[width=\textwidth]{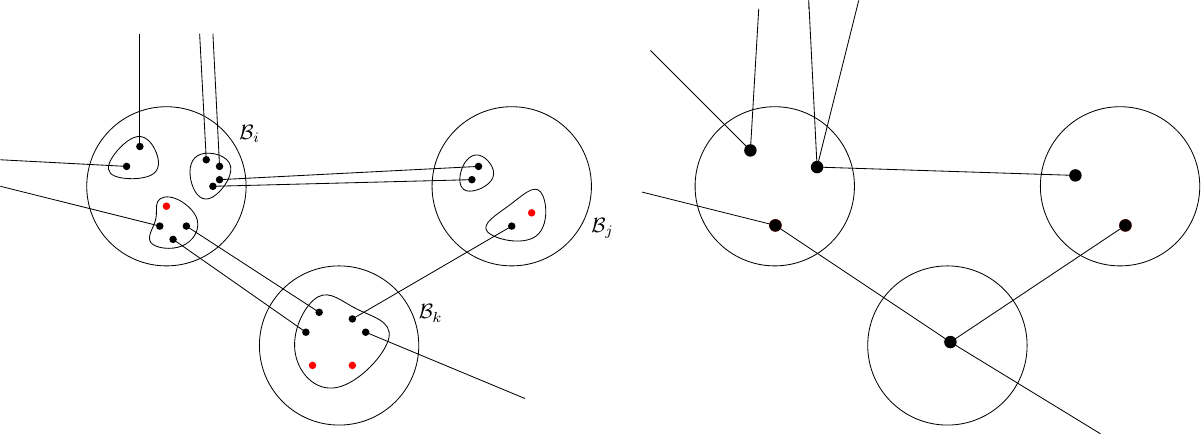}
    \caption{An example of a tangled current $(\n,\ct)$ with a moment function $A$ satisfying $A_i=1, A_j=1$, and $A_k=2$. On the left picture, the circles represent the blocks $\cB(\n,A)$; the small ellipses inside represent the elements of the partitions $\ct$. Notice that $i,j$ are sources of $\n$, and the red vertices correspond to the vertices associated with $A$. On the right picture, the multigraph $\cH(\n,\ct,A)$ is obtained by identifying vertices in the same ellipses.}
    \label{fig:tanglings}
\end{figure} 

We define $\Omega_{\Lambda}^\cT$ to be the set of tangled currents on $\Lambda$. Fix $\beta>0$. Before recalling the switching lemma for $\varphi^4$, we need a definition.

\begin{defn} Let $B\in \cM(\Lambda)$. Define $\cF_B^{\Lambda}$ to be the subset of $\Omega_{\Lambda}^\cT$ consisting of all tangled currents $(\n,\ct)$ for which every connected component intersects $\lbrace xb(k): \: x\in \Lambda, \: 1\leq k \leq B_x\rbrace$ an even number of times. In words, $\cF_B^{\Lambda}$ consists of all tangled currents that pair the elements of $B$ in its associated multigraph. The set $\cF_B^{\Lambda}$ is measurable with respect to the natural $\sigma$-algebra on $\Omega_{\Lambda}^\cT$.
\end{defn}

We let 
\begin{equation}\label{eq:weights_def}
w_{A,\beta}(\n)
:=
\prod_{\lbrace x,y\rbrace\in E}\dfrac{\beta^{\n_{x,y}}}{\n_{x,y}!}\left\langle \varphi^{\Delta_\n(x)+A_x}\right\rangle_{0,g,a}.
\end{equation}
 
\begin{thm}[Switching lemma for $\varphi^4$, {\cite[Theorem~3.11]{gunaratnam2022random}}]\label{thm: switching lemma}
Let $A,B \in \cM(\Lambda)$. For every $\n_1,\n_2 \in \Omega_\Lambda$ satisfying $\partial \n_1 = \partial A$ and $\partial \n_2 = \partial B$ and for every $z\in \Lambda$, there are probability measures $\rho^{\n_1,\n_2}_{z,A,B}=\rho^{\n_1,\n_2}_{z,A,B,g,a}$ on $\cT_{\n,A,B}(z)$ such that the following holds. 
For any measurable and bounded $F:\Omega_{\Lambda}^{\cT} \rightarrow \RR$, 
\begin{equs}
\sum_{\substack{\partial \n_1 = \partial A \\ \partial \n_2 = \partial B}} 
&w_{A,\beta}(\n_1) w_{B,\beta}(\n_2) \rho_{A,B}^{\n_1,\n_2}\big[ F(\n_1+\n_2,\ct) \big]
\\
&=
\sum_{\substack{\partial \n_1 = \partial(A+B) \\ \partial \n_2 = \emptyset}} 
w_{A+B,\beta}(\n_1)w_{\emptyset,\beta}(\n_2)  \rho_{A+B,\emptyset}^{\n_1,\n_2} \big[ F(\n_1 + \n_2, \ct) \mathbbm{1}_{\cF_{B}^{\Lambda}} \big],
\end{equs}
where
\begin{equs}
\rho^{\n_1,\n_2}_{A,B}
:=
\bigotimes_{z \in \Lambda} \rho^{\n_1, \n_2}_{z,A,B},\text{ and }\rho^{\n_1,\n_2}_{A+B,\emptyset}:=\bigotimes_{z \in \Lambda} \rho^{\n_1, \n_2}_{z,A+B,\emptyset},
\end{equs}
are measures on $\cT_{\n,A,B}$ and $\cT_{\n,A+B,\emptyset}$ respectively, and $\cF_B^{\Lambda}$ is the event defined above.
\end{thm}
In \cite{gunaratnam2022random}, the probability measures $\rho_{z,A,B}^{\n_1,\n_2}$ were obtained by taking subsequential limits of certain measures defined below. Our main result implies full convergence of these measures. We will need the following definitions.

If $S_1,S_2$ are two even and disjoint subsets of $V_n$, we let $\tilde{\rho}_{n,\lambda}^{S_1,S_2}$ be the measure $\mathbb P^{S_1,S_2}_{n,\lambda}$ quotiented by the relation:
\begin{equs}
    \n\sim \m \Longleftrightarrow \mathfrak{P}(\n)=\mathfrak{P}(\m).
\end{equs}
In words, $\tilde{\rho}_{n,\lambda}^{S_1,S_2}$ is the law of $\mathfrak{P}(\n)$ for $\n$ sampled according to $\mathbb P^{S_1,S_2}_{n,\lambda}$. Note that the measure $\tilde{\rho}_{n,\lambda}^{S_1,S_2}$ is supported on the set $\mathcal{P}_{\mathrm{even}}(S_1,S_2)$ which consists of even partitions $\mathsf{P}=\lbrace P_1,\ldots, P_r\rbrace$ of $S_1\sqcup S_2$ which are compatible\footnote{Note also that $\mathcal{P}_{\mathrm{even}}(S_1,\emptyset)=\mathcal{P}_{\mathrm{even}}(S_1)$.} with both $S_1$ and $S_2$ in the sense that: for all $1\leq i \leq r$, $|P_i\cap S_1|$ and $|P_i\cap S_2|$ are even.

For $x\in V$, the measure $\rho^{\n_1,\n_2}_{x,A,B}$ is defined in \cite{gunaratnam2022random} as a subsequential limit of the sequence of measures $(\tilde{\rho}^{S_1,S_2}_{n,\lambda})_{n\geq 1}$ where $|S_1|=A_x+\Delta_{\n_1}(x)$ and $|S_2|=B_x+\Delta_{\n_2}(x)$ (and $S_1,S_2$ are disjoint).

As mentioned in \eqref{eq: switching intro}, the switching lemma (for the Ising model), gives a connection between $\tilde{\rho}_{n,\lambda}^{S_1,S_2}$ and $\tilde{\rho}_{n,\lambda}^{S_1\sqcup S_2,\emptyset}$. Indeed, notice that the event $\mathcal{F}_{S_2}$ defined below \eqref{eq: switching intro} satisfies $\lbrace \n, \: \mathfrak{P}(\n)=\mathsf{P}\rbrace\subset \mathcal{F}_{S_2}$ for any $\mathsf{P}\in \mathcal{P}_{\mathrm{even}}(S_1,S_2)$. As a result, for any such $\mathsf{P}$,
\begin{equs}\label{eq: application of the switching to connect the rho measures}
    \tilde{\rho}_{n,\lambda}^{S_1,S_2}[\mathsf{P}]=\frac{\langle \sigma_{S_1\sqcup S_2}\rangle_{n,\lambda}}{\langle \sigma_{S_1}\rangle_{n,\lambda}\langle \sigma_{S_2}\rangle_{n,\lambda}}\tilde{\rho}_{n,\lambda}^{S_1\sqcup S_2,\emptyset}[\mathsf{P}].
\end{equs}
\begin{prop} Let $\lambda\in \mathbb R$. For any $S_1,S_2$ even and disjoint subsets of $V_n$. The sequence of measures $(\tilde{\rho}^{S_1,S_2}_{n,\lambda})_{n\geq 1}$ converges in the sense of distribution.
\end{prop}
\begin{proof} Since the sequence of measures $(\tilde{\rho}^{S_1,S_2}_{n,\lambda})_{n\geq 1}$ lives on the finite space $\mathcal{P}_{\mathrm{even}}(S_1,S_2)$ we just need to prove the convergence for $\mathsf{P}\in \mathcal{P}_{\mathrm{even}}(S_1,S_2)$ of $(\tilde{\rho}^{S_1,S_2}_{n,\lambda}[\mathsf{P}])_{n\geq 1}$. Using Corollary \ref{cor: cor 2}, Lemma \ref{lem: GS}, and \eqref{eq: application of the switching to connect the rho measures}, we find that
\begin{equs}
    \lim_{n\rightarrow\infty} \tilde{\rho}^{S_1,S_2}_{n,\lambda}[\mathsf{P}]=\frac{\langle \varphi^{|S_1|+|S_2|}\rangle_0}{\langle \varphi^{|S_1|}\rangle_0\langle \varphi^{|S_2|}\rangle_0}\mathbb P[P^{\mathsf{d}}_{k,\lambda}=\mathsf{P}],
\end{equs}
where $\langle \cdot\rangle_0=\langle \cdot\rangle_{0,g,a(\lambda)}$ was defined in Lemma \ref{lem: GS}, and $k=(|S_1|+|S_2|)/2$.
\end{proof}
As mentioned in \cite{gunaratnam2022random}, we can also give a probabilistic interpretation of Theorem \ref{thm: switching lemma}. We begin by defining the natural probability measures on the set of tangled currents.
\begin{defn}
Let $A,B\in \cM(\Lambda)$. We define a measure $\P_{\Lambda,g,a\beta}^{A,B}$ on the set of tangled currents on $\Lambda$ with source set $A+B$ as follows: for $(\n,\ct)\in \Omega_{\Lambda,\cT}$,
\begin{equs}\label{eq: def P for phi4}
    \P^{A,B}_{\Lambda,g,a,\beta}[(\n,t)]=\mathbbm{1}_{\partial \n=\partial(A+B)}\frac{\sum_{\substack{\partial \n_1=\partial A\\ \partial\n_2=\partial B}}w_{A,\beta}(\n_1)w_{B,\beta}(\n_2)\mathbbm{1}_{\n_1+\n_2=\n}\rho_{A,B,g,a}^{\n_1,\n_2}[(\n,\ct)]}{\sum_{\substack{\partial \n_1=\partial A\\ \partial \n_2=\partial B}}w_{A,\beta}(\n_1)w_{B,\beta}(\n_2)}.
\end{equs}
We write $\mathbf{E}^A_{\Lambda,\beta}$ for the expectation with respect to the measure $\mathbf{P}^A_{\Lambda,\beta}$.
\end{defn}

\begin{cor}[Probabilistic version of the switching lemma]\label{thm: probabilistic switching}
Let $A,B \in \cM(\Lambda)$. For any measurable and bounded $F:\Omega_{\Lambda,\cT}\rightarrow \mathbb R$,
\begin{equs}
    \frac{\langle \varphi_A\rangle_{\Lambda,g,a,\beta}\langle \varphi_B\rangle_{\Lambda,g,a,\beta}}{\langle \varphi_{A+B}\rangle_{\Lambda,g,a,\beta}}\mathbf{E}^{A,B}_{\Lambda,\beta}[F(\n,\ct)]=\mathbf{E}^{A+B,\emptyset}_{\Lambda,g,a,\beta}[F(\n,\ct)\mathbbm{1}_{\cF^{\Lambda}_B}],
\end{equs}
where for $S\in \mathcal{M}(\Lambda)$, $\langle \varphi_S\rangle_{\Lambda,g,a,\beta}:=\langle \prod_{x\in \Lambda}\varphi_x^{S_x}\rangle_{\Lambda,g,a,\beta}$.
\end{cor}

\subsection{An alternative proof of Wick's law}

Let $U^{\n_1,\n_2}_{x,A,B}$ be the uniform measure on pairings in $\cT_{\n, A,B}(x)$.

\begin{prop}\label{prop: cv of the tanglings to unif on pairings}
Let $a>0$ and $x\in V$. As $g$ tends to $0$, $\rho^{\n_1, \n_2}_{x,A,B,g,a}$ converges in distribution to $U^{\n_1,\n_2}_{x,A,B}$.
\end{prop}
\begin{proof}
Fix $S_1, S_2$ even and disjoint subsets of $V_n$ of size $A_x+\Delta_{\n_1}(x)$ and $B_x+\Delta_{\n_2}(x)$, respectively. Let $\mu^{\n_1,\n_2}_{x,A,B}$ be any limiting measure as $g$ tends to $0$. Due to the invariance of $\mathbb{P}^{S_1,S_2}_{n,\lambda}$ with respect to permuting the elements of $S_1$ and the elements of $S_2$,  $\mu^{\n_1,\n_2}_{x,A,B}$ also has this property, and the desired result will follow once we prove that $\mu^{\n_1,\n_2}_{x,A,B}$ is supported on (admissible) pairings. 

To this end, note that if there are no vertices of degree $4$ (or more) under $\mathbb{P}^{S_1,S_2}_{n,\lambda}$, then $\mathfrak{P}(\n)$ consists of pairs, see the discussion above Definition \ref{def: backbone multigraph}. Let us estimate the probability of having a vertex of degree at least $4$. By \eqref{eq: application of the switching to connect the rho measures},  it is sufficient to study this question for the measure $\mathbb P_{n,\lambda}^{S_1\sqcup S_2,\emptyset}$. Let $S:=S_1\sqcup S_2$.

By Lemma~\ref{lem: degree}~$(ii)$, all sources have degree $1$ with high probability, and it remains to consider the degree of the vertices in $V_n\setminus S$. 
Let $i\in V_n\setminus S$. If $\Delta_\n(i)\geq 4$ then one of the following occurs: $(i)$ $\Delta_{\n_1}(i)\geq 4$ or $\Delta_{\n_2}(i)\geq 4$, $(ii)$ $\Delta_{\n_1}(i)\geq 2$ and $\Delta_{\n_2}(i)\geq 2$. The contribution coming from case $(i)$ is handled by \eqref{eq: degree 4}. The contribution coming from $(ii)$ is handled by noticing that for $j,k\notin S$,
\begin{eqnarray*}
\mathbb{P}^{S,\emptyset}_{n,\lambda}\left[\n_1(i,j)\geq 1, \: \n_2(i,k)\geq 1\right]&\leq& 
(1+o(1))d_n(\lambda)^2 (c_n n)^{-4} \frac{\langle \varphi^{|S|+2}\rangle_{0} \langle \varphi^{2}\rangle_{0}}{\langle \varphi^{|S|}\rangle_{0}}
\\&=& (1+o(1))\tilde{g}^4n^{-3}\frac{\langle \varphi^{|S|+2}\rangle_0 \langle \varphi^{2}\rangle_{0}}{\langle \varphi^{|S|}\rangle_0},  
\end{eqnarray*}
which implies that
\begin{equs}\label{eq: degree 2+2}
\mathbb{E}^{S,\emptyset}_{n,\lambda}\left[\sum_{i\notin S}\mathbb 1\left[\Delta_{\n_1}(i)\geq 2, \: \Delta_{\n_2}(i)\geq 2\right]\right] \leq (1+o(1))\tilde{g}^4\frac{\langle \varphi^{|S|+2}\rangle_0 \langle \varphi^{2}\rangle_{0}}{\langle \varphi^{|S|}\rangle_0}.    
\end{equs}
Hence, one can find $C>0$ (which does not depend on $g$) such that, 
\begin{equs}
    \limsup_{n\rightarrow\infty}\mathbb E_{n,\lambda}^{S,\emptyset}\Big[|\lbrace i\in V_n, \: \Delta_\n(i)\geq 4\rbrace|\Big]\leq C\tilde{g}^4\left(\frac{\langle \varphi^{|S|+2}\rangle_{0,g,a} \langle \varphi^{2}\rangle_{0,g,a}}{\langle \varphi^{|S|}\rangle_{0,g,a}}+\frac{\langle \varphi^{|S|+4}\rangle_{0,g,a}}{\langle \varphi^{|S|}\rangle_{0,g,a}}+\langle \varphi^{4}\rangle_{0,g,a}\right).
\end{equs}
If $p\geq 0$, as $g$ tends to $0$, $\langle \varphi^{2p}\rangle_{0,g,a}$ tends to $\langle \varphi^{2p}\rangle_{0,a}$.
As a consequence, using Markov's inequality,
\begin{equs}
\lim_{g\to 0}\limsup_{n\to\infty}\mathbb{P}^{S,\emptyset}_{n,\lambda}\left[\exists \: i\in V_n, \: \Delta_{\n}(i)\geq 4\right]=0.    
\end{equs}
The desired result follows.
\end{proof}
Combining Theorem \ref{thm: switching lemma} with the above result yields a new proof of Wick's law for the GFF.
\begin{cor}[Wick's law] Let $\Lambda=(V,E)$ be a finite graph. Let $a>0$ and $\beta\geq 0$ be such that $a>4d\beta$. Consider the Gaussian free field at parameter $a$ and inverse temperature $\beta$ on $\Lambda$. Then for any $p\geq 1$, for any $(i_k)_{1\leq k \leq 2p}\in V^{2k}$, one has,
\begin{equs}
    \mathbb E^{\mathsf{GFF}}_{\Lambda,a,\beta}\left[\prod_{k=1}^{2p}\varphi_{i_k}\right]=\sum_{\pi}\prod_{k=1}^p\mathbb E_{\Lambda,a,\beta}^{\mathsf{GFF}}\left[\varphi_{i_{\pi(2k-1)}}\varphi_{i_{\pi(2k)}}\right],
\end{equs}
where the sum is over pairings $\pi$ of $\lbrace 1,\ldots,2p\rbrace$.
\end{cor}
\begin{proof} For $x\in \Lambda$, set $A_x:=|\lbrace k, \: i_k=x\rbrace|$. This defines an $A\in \mathcal{M}(\Lambda)$. Fix $u\in \Lambda$ such that $A_u>0$. If $v\in \Lambda$ satisfies $A_v>0$, define $C=C_{u,v}:= \mathbb{1}_u+\mathbb{1}_v$ and $B=B_{u,v}:=A-C_{u,v}$.

By Corollary \ref{thm: probabilistic switching} applied to $F\equiv 1$: for all $1\leq j\leq A_v$ (except if $u=v$ in which case\footnote{For this case to occur, one must have $A_u\geq 2$.} we require $j\neq 1$),
\begin{equs}\label{eq: wick1}
    \frac{\langle \varphi_B\rangle^{\mathsf{GFF}}_{\Lambda,a,\beta}\langle \varphi_C\rangle^{\mathsf{GFF}}_{\Lambda,a,\beta}}{\langle \varphi_{A}\rangle^{\mathsf{GFF}}_{\Lambda,a,\beta}}=\mathbf{P}^{A,\emptyset}_{\Lambda,a,\beta}[\cF^{\Lambda}_C]=\mathbf{P}^{A,\emptyset}_{\Lambda,a,\beta}[ua(1)\overset{\cH(\n,\ct,A,\emptyset)}{\longleftrightarrow} va(j)],
\end{equs}
where $\mathbf{P}^{A,\emptyset}_{\Lambda,a,\beta}:=\lim_{g\to 0}\mathbf{P}^{A,\emptyset}_{\Lambda,g,a,\beta}$. In particular, by Proposition \ref{prop: cv of the tanglings to unif on pairings} this measure is defined similarly as in \eqref{eq: def P for phi4} replacing the measures $\rho_{x,A,\emptyset,g,a}^{\n_1+\n_2}$ by $U^{\n_1,\n_2}_{x,A,\emptyset}$.
Now, since $U^{\n_1,\n_2}_{x,A,\emptyset}$ is supported on pairings, almost surely under $\cF^{\Lambda}_B$, all vertices in $\cH(\n,\ct, A,\emptyset)$ have degree $2$ except for the sources that have degree $1$. This implies that $ua(k)$ is connected to only one element of $\mathcal S_{u}:=\lbrace va(t), \: 1\leq t \leq A_v, \: v\in \Lambda\setminus \lbrace u\rbrace\rbrace\cup \lbrace ua(t),\: 2\leq t \leq A_u\rbrace$ in $\cH(\n,\ct,A,\emptyset)$. Thus, 
\begin{equs}\label{eq: wick2}
    \sum_{s\in \mathcal{S}_{u}}\mathbf{P}^{A,\emptyset}_{\Lambda,a,\beta}[ua(1) \overset{\cH(\n,\ct,A,\emptyset)}{\longleftrightarrow} s]=1.
\end{equs}
As a result, summing \eqref{eq: wick1} over $va(j)\in \mathcal{S}_u$ and using \eqref{eq: wick2} yields
\begin{equs}
    \langle \varphi_A\rangle^{\mathsf{GFF}}_{\Lambda,a,\beta}=\sum_{v\in \Lambda}\sum_{j=1+\delta_u(v)}^{A_v}\langle \varphi_{B_{u,v}}\rangle^{\mathsf{GFF}}_{\Lambda,a,\beta}\langle \varphi_{C_{u,v}}\rangle^{\mathsf{GFF}}_{\Lambda,a,\beta},
\end{equs}
which implies the result by induction.
\end{proof}
\subsection{Recovering the switching lemma for the Ising model}
Below, we call $\mathsf{ECT}$ (for ``Everybody Connected Together'') the event that $\mathfrak{P}(\n)=\lbrace \sn\rbrace$.
\begin{prop} Let $a=-2g$. Then, as $g$ tends to $\infty$, $\rho_{x,A,B}^{\n_1,\n_2}$ converges in distribution to the Dirac measure $\delta_{\mathsf{ECT}}$.
\end{prop}
\begin{proof} First, observe that if $p\geq 0$, as $g$ tends to $\infty$, $\langle \varphi^{2p}\rangle_{0,g,-2g}$ tends to $1$. Fix $S_1, S_2$ even and disjoint subsets of $V_n$ of size $A_x+\Delta_{\n_1}(x)$ and $B_x+\Delta_{\n_2}(x)$, respectively. Using \eqref{eq: application of the switching to connect the rho measures} and the observation above, we see that for any $\mathsf{P}\in \mathcal{P}_{\mathrm{even}}(S_1,S_2)$,
\begin{equs}
    \lim_{g\rightarrow \infty}\rho_{x,A,B,g,-2g}^{\n_1,\n_2}[\mathsf{P}]=\lim_{g\rightarrow \infty}\lim_{n\rightarrow\infty}\mathbb P_{n,\lambda(g)}^{S,\emptyset}[\mathfrak{P}(\n_1+\n_2)=\mathsf{P}],
\end{equs}
where $S=S_1\sqcup S_2$ and $\lambda(g)$ is such that $a(\lambda(g))=-2g$. We can always assume that $S=\lbrace 1,\ldots,|S|\rbrace$. Clearly, if $\mathsf{P}\neq \lbrace S\rbrace$, there exists $1\leq i < j \leq |S|$ such that $i$ is not connected to $j$ in $\n_1+\n_2$. As a result, for such a $\mathsf{P}$,
\begin{equs}
    \mathbb P_{n,\lambda(g)}^{S,\emptyset}[\mathfrak{P}(\n_1+\n_2)=\mathsf{P}]\leq \sum_{1\leq i <j \leq |S|}\left(1-\mathbb P_{n,\lambda(g)}^{S,\emptyset}[i \overset{\n_1+\n_2}{\longleftrightarrow} j]\right).
\end{equs}
Using the switching lemma for the Ising model \eqref{eq: switching 1} and Lemma \ref{lem: GS}, we get that 
\begin{equs}
    \mathbb P_{n,\lambda(g)}^{S,\emptyset}[i \overset{\n_1+\n_2}{\longleftrightarrow} j]=(1+o(1))\frac{\langle \varphi^{|S|-2}\rangle_{0,g,-2g}\langle \varphi^{2}\rangle_{0,g,-2g}}{\langle \varphi^{|S|}\rangle_{0,g,-2g}}.
\end{equs}
This readily implies that for $\mathsf{P}\neq \lbrace S\rbrace$,
\begin{equs}
    \lim_{g\rightarrow \infty}\lim_{n\rightarrow\infty}\mathbb P_{n,\lambda(g)}^{S,\emptyset}[\mathfrak{P}(\n_1+\n_2)=\mathsf{P}]=0
\end{equs}
and the result follows.
\end{proof}
The above result immediately implies that one can recover the switching lemma for the Ising model when taking the Ising limit in the identity of Theorem \ref{thm: switching lemma}. Indeed, the measures $\rho_{x,A,B}^{\n_1,\n_2}$ converge to $\delta_{\mathsf{ECT}}$ which has as a consequence that the vertex set of the multigraph $\cH(\n,\ct,A,B)$ coincides with $V$, hence we can identify $\cH(\n,\ct,A,B)$ with the current $\n$.

\bibliographystyle{alpha}
\bibliography{ref}

\begin{thebibliography}{ADCTW19}

\bibitem[ABBG12]{addario2012continuum}
L.~Addario-Berry, N.~Broutin, and C.~Goldschmidt.
\newblock The continuum limit of critical random graphs.
\newblock {\em Probability Theory and Related Fields},
  \textbf{152}(3-4):367--406, 2012.

\bibitem[ABF87]{aizenman1987phase}
M.~Aizenman, D.J. Barsky, and R.~Fern{\'a}ndez.
\newblock The phase transition in a general class of {I}sing-type models is
  sharp.
\newblock {\em Journal of Statistical Physics}, \textup{\textbf{47}}:343--374,
  1987.

\bibitem[ADC21]{ADC}
M.~Aizenman and H.~Duminil-Copin.
\newblock Marginal triviality of the scaling limits of critical $4d$ {I}sing
  and $\varphi^4_4$ models.
\newblock {\em Annals of Mathematics}, \textup{\textbf{194}}(1):163--235, 2021.

\bibitem[ADCTW19]{aizenman2019emergent}
H.~Aizenman, H.~Duminil-Copin, V.~Tassion, and S.~Warzel.
\newblock Emergent planarity in two-dimensional {I}sing models with
  finite-range interactions.
\newblock {\em Inventiones Mathematicae}, \textup{\textbf{216}}(3):661--743,
  2019.

\bibitem[AF86]{aizenman1986critical}
M.~Aizenman and R.~Fern{\'a}ndez.
\newblock On the critical behavior of the magnetization in high-dimensional
  {I}sing models.
\newblock {\em Journal of Statistical Physics}, \textbf{44}(3-4):393--454,
  1986.

\bibitem[AG83]{aizenman1983renormalized}
M.~Aizenman and R.~Graham.
\newblock On the renormalized coupling constant and the susceptibility in
  $\varphi^4_4$ field theory and the {I}sing model in four dimensions.
\newblock {\em Nuclear Physics B}, \textup{\textbf{225}}(2):261--288, 1983.

\bibitem[Aiz82]{A}
M.~Aizenman.
\newblock Geometric analysis of $\varphi^4$ fields and {I}sing models. parts
  {I} and {II}.
\newblock {\em Communications in Mathematical Physics},
  \textup{\textbf{86}}(1):1--48, 1982.

\bibitem[Ald97]{aldous1997brownian}
D.~Aldous.
\newblock Brownian excursions, critical random graphs and the multiplicative
  coalescent.
\newblock {\em The Annals of Probability}, \textbf{25}:812--854, 1997.

\bibitem[AP00]{aldous2000random}
D.J. Aldous and B.~Pittel.
\newblock On a random graph with immigrating vertices: {E}mergence of the giant
  component.
\newblock {\em Random Structures \& Algorithms}, \textbf{17}(2):79--102, 2000.

\bibitem[BGJ96]{bollobas1996random}
B.~Bollob{\'a}s, G.~Grimmett, and S.~Janson.
\newblock The random-cluster model on the complete graph.
\newblock {\em Probability Theory and Related Fields}, \textbf{104}:283--317,
  1996.

\bibitem[DC17]{duminil2017lectures}
H.~Duminil-Copin.
\newblock Lectures on the ising and potts models on the hypercubic lattice.
\newblock In {\em PIMS-CRM Summer School in Probability}, pages 35--161.
  Springer, 2017.

\bibitem[DCT16]{DCT}
H.~Duminil-Copin and V.~Tassion.
\newblock A new proof of the sharpness of the phase transition for {B}ernoulli
  percolation and the {I}sing model.
\newblock {\em Communications in Mathematical Physics},
  \textup{\textbf{343}}(2):725--745, 2016.

\bibitem[FV17]{friedli2017statistical}
S.~Friedli and Y.~Velenik.
\newblock {\em Statistical mechanics of lattice systems: a concrete
  mathematical introduction}.
\newblock Cambridge University Press, 2017.

\bibitem[GHS70]{GHS}
R.B. Griffiths, C.A. Hurst, and S.~Sherman.
\newblock Concavity of magnetization of an {I}sing ferromagnet in a positive
  external field.
\newblock {\em Journal of Mathematical Physics},
  \textup{\textbf{11}}(3):790--795, 1970.

\bibitem[GPPS22]{gunaratnam2022random}
T.~S. Gunaratnam, C.~Panagiotis, R.~Panis, and F.~Severo.
\newblock Random tangled currents for $\varphi^4$: translation invariant gibbs
  measures and continuity of the phase transition.
\newblock {\em arXiv:2211.00319}, 2022.

\bibitem[Hof16]{van2016random}
R.~Van~Der Hofstad.
\newblock {\em Random graphs and complex networks}, volume \textbf{43}.
\newblock Cambridge University Press, 2016.

\bibitem[Hus54]{husimi1953statistical}
K.~Husimi.
\newblock Statistical mechanics of condensation.
\newblock In {\em Proceedings of the International Conference of Theoretical
  Physics}, 1954.

\bibitem[NP10]{nachmias2010critical}
A.~Nachmias and Y.~Peres.
\newblock Critical percolation on random regular graphs.
\newblock {\em Random Structures \& Algorithms}, \textbf{36}(2):111--148, 2010.

\bibitem[Pan23]{panis2023triviality}
R.~Panis.
\newblock Triviality of the scaling limits of critical {I}sing and $\varphi^4$
  models with effective dimension at least four.
\newblock {\em arXiv:2309.05797}, 2023.

\bibitem[Rao20]{raoufi2020translation}
A.~Raoufi.
\newblock {Translation-invariant {G}ibbs states for the {I}sing model: general
  setting}.
\newblock {\em The Annals of Probability}, \textbf{48}(2):760--777, 2020.

\bibitem[SG73]{simon1973varphi}
B.~Simon and R.B. Griffiths.
\newblock The $\varphi^4_2$ field theory as a classical {I}sing model.
\newblock {\em Communications in Mathematical Physics},
  \textup{\textbf{33}}:145--164, 1973.

\bibitem[Tem54]{temperley1954mayer}
H.N.V. Temperley.
\newblock The {M}ayer theory of condensation tested against a simple model of
  the imperfect gas.
\newblock {\em Proceedings of the Physical Society. Section A},
  \textbf{67}(3):233, 1954.

\end{thebibliography}

\appendix
\section{Explicit Griffiths-Simon approximation for $\phi^4$}\label{appendix: proof of GS}
In this appendix, we obtain asymptotic estimates for the partition functions of the Ising and random current models on the complete graph. Along the way, we reprove the main result of \cite{simon1973varphi}.

\begin{prop}\label{prop: gs approx}
Let $\lambda\in \mathbb R$. Let $(\sigma_1,\dots, \sigma_n)$ be distributed as spins of the Ising model on the complete graph $K_n$ with inverse temperature $d_n(\lambda)$. Then, one has 
\[
\frac{\sigma_1+\dots+\sigma_n}{n^{3/4}}\underset{n\rightarrow \infty}{\longrightarrow} \frac{1}{z(\lambda)}\: e^{-\frac{1}{12}s^4-\frac{\lambda}{2}s^2}\: \mathrm{d}s,
\]
where the convergence is in distribution and $z(\lambda):=\int_{\mathbb{R}}  e^{-\frac{1}{12}s^4-\frac{\lambda}{2}s^2}\mathrm{d}s $ is a renormalising factor.
Moreover, for the partition function of the Ising model on $K_n$ we have 
\[
\mathcal{Z}_n(\lambda)= (1+o(1)) \frac{2^n{n^{1/4}}}{\sqrt{2\pi e}} \int_{\mathbb{R}}  e^{-\frac{1}{12}s^4-\frac{\lambda}{2}s^2}\mathrm{d}s.
\]
\end{prop}
\begin{proof}
We assume that $n$ is even, the other case is similar. We denote by $\mathbf P_{n,\lambda}$ the measure of the Ising model on $K_n$ with inverse temperature $d_n(\lambda)$. Note that
\[
2\sum_{1\leq i < j \leq n} \sigma_i\sigma_j=\left(\sum_{i=1}^n \sigma_i\right)^2-n,
\] 
so that, for any $m\in [0, n/2]$ we have, 
\[
    \mathbf{P}_{n,\lambda}\left[\sum_{i=1}^n \sigma_i = 2m\right]=\frac{1}{\mathcal{Z}_n(\lambda)}\binom{n}{n/2+m} \exp{\left(d_n(\lambda)\frac{4m^2-n}{2}\right)}=\frac{\exp(-d_n(\lambda) \frac{n}{2})}{\mathcal{Z}_n(\lambda)}f_n(m),
\]
where $f_n(m):=\binom{n}{n/2+m} \exp{\left(2 d_n(\lambda) m^2\right)}$. We have
\begin{equs}
\log{\frac{f_n(m)}{f_n(m-1)}} 
&= \log{\frac{n/2-m+1}{n/2+m}} + (1-\lambda/\sqrt{n}) \frac{4m-2}{n} \\
&= \log{\frac{1-2m/n}{1+2m/n}}+ (1-\lambda/\sqrt{n}) \frac{4m}{n}+O\left(\frac{1}{n}\right).
\end{equs}
Expanding in powers of $m/n$ and using that $\log{(1-2t)}-\log{(1+2t)}=-4t-16t^3/3+O(t^5)$, we observe that the first order term cancels and we get
\[
\log{\frac{f_n(m)}{f_n(m-1)}} = -\frac{16m^3}{3 n^3} - \lambda\frac{4m}{n\sqrt{n}} +O\left(\frac{m^5}{n^5}+\frac{1}{n}\right).
\]
By telescoping, we obtain that
\begin{equs}
\log\frac{f_n(m)}{f_n(0)}=-\frac{4m^4}{3 n^3} - \lambda\frac{2m^2}{n\sqrt{n}} +O\left(\frac{m^6}{n^5}+\frac{m}{n}\right).
\end{equs}
By symmetry, the above holds also for negative $m$ up to an error term $O\left(\frac{m^6}{n^5}+\frac{|m|}{n}\right)$.
Letting $g(x)=\exp\left(-\frac{4}{3}x^4-2\lambda x^2\right)$, for any $a,b\in \mathbb{R}$ with $a<b$ we have
\begin{equs}
    \mathbf{P}_{n,\lambda}\left[an^{3/4}\leq \sum_{i=1}^n \sigma_i \leq bn^{3/4}\right]&=(1+o(1))\frac{\exp(-d_n(\lambda) \frac{n}{2}) f_n(0)}{\mathcal{Z}_n(\lambda)}\sum_{\frac{a}{2}n^{3/4}\leq m \leq \frac{b} {2}n^{3/4}} g\left(\frac{m}{n^{3/4}}\right)
    \\
    &=(1+o(1))\frac{e^{-1/2} f_n(0)}{\mathcal{Z}_n(\lambda)} n^{3/4}\int_{\frac{a}{2}}^{\frac{b}{2}} g\left(\lambda\right)\mathrm{d}\lambda 
    \\
    &=(1+o(1))\:\frac{e^{-1/2} f_n(0)}{2 \mathcal{Z}_n(\lambda)}\: n^{3/4}\int_a^b  e^{-\frac{1}{12}s^4-\frac{\lambda}{2}s^2} \mathrm{d}s.
\end{equs}

It remains to estimate $\mathcal{Z}_n(\lambda)$, and for that we need to handle the case where $m$ is much larger than $n^{3/4}$. Since all terms in the Taylor expansion of $\log{(1-2t)}-\log{(1+2t)}$ are non-positive it follows that
\[
\log{\frac{f_n(m)}{f_n(m-1)}} \leq -\frac{16m^3}{3n^3} -\lambda\frac{4m}{n^{3/2}} + O\left(\frac{1}{n}\right).
\]
Thus, there exists $C>0$ such that $f_n(m)\leq Cf_n(0) g\left(\frac{m}{n^{3/4}}\right)$,
from which we can conclude by arguing as above that
\begin{equs}
\mathcal{Z}_n(\lambda)= (1+o(1))\: \frac{e^{-1/2} f_n(0)}{2}\: n^{3/4}\int_{\mathbb{R}}  e^{-\frac{1}{12}s^4-\frac{\lambda}{2}s^2}\: \mathrm{d}s. 
\end{equs}
Sending $n$ to infinity we get the desired convergence in distribution.

Finally, we have $f_n(0) = (1+o(1))\: 2^n/\sqrt{\pi n/2}$ by Stirling's approximation, hence
\[
\mathcal{Z}_n(\lambda)=(1+o(1)) \: \frac{2^n n^{1/4}}{\sqrt{2\pi e}} \int_{\mathbb{R}}  e^{-\frac{1}{12}s^4-\frac{\lambda}{2}s^2}\mathrm{d}s,
\]
as desired.
\end{proof}
Recall that for $g>0$ and $\lambda\in \mathbb R$, the measure $\langle \cdot\rangle_0$ is the average with respect to $\rho_{g,a(\lambda)}$ with $a(\lambda)=\frac{\lambda\tilde{g}^{2}}{2}$.
\begin{prop}[Estimate of the partition function of the current]\label{prop: estimating the partition function}
Let $\lambda\in \mathbb R$, $g>0$, and $k\geq 1$. Set $S=\lbrace 1,\ldots, 2k\rbrace$. Then, 
\[
Z^S_n(\lambda)=(1+o(1))\tilde{g}^{2k}\frac{\left\langle\varphi^{2k}\right\rangle_0}{\sqrt{2\pi e}} \: n^{1/4-k/2} \int_{\mathbb{R}}  e^{-\frac{1}{12}s^4-\frac{\lambda}{2}s^2}\: \mathrm{d}s.
\]
Moreover, for every $n\geq 1$, $Z^S_n(\lambda)$ is decreasing as a function of $\lambda\in (-\infty,\sqrt{n}]$.
\end{prop}
\begin{proof}
Recall that by Lemma \ref{lem: GS}
\begin{equs}
\frac{Z^S_n(\lambda)}{Z_n(\lambda)}=\langle \sigma_S\rangle_{n,\lambda}=(1+o(1))\:(c_n n)^{-2k} \left\langle\varphi^{2k}\right\rangle_0,
\end{equs}
and as mentioned in \eqref{eq: link between partition functions},
$Z_n(\lambda)=2^{-n}\mathcal{Z}_n(\lambda)$.
Thus, by Proposition \ref{prop: gs approx}, and since $(c_nn)^2=\tilde{g}^{-2}\sqrt{n}$,
\begin{equs}
Z^S_n(\lambda)&=(1+o(1))\: (c_n n)^{-2k} 
\frac{\left\langle\varphi^{2k}\right\rangle_0}{\sqrt{2\pi e}} \: n^{1/4} \int_{\mathbb{R}}  e^{-\frac{1}{12}s^4-\frac{\lambda}{2}s^2}\: \mathrm{d}s \\
&=(1+o(1))
\tilde{g}^{2k}\frac{\left\langle\varphi^{2k}\right\rangle_0}{\sqrt{2\pi e}} \: n^{1/4-k/2} \int_{\mathbb{R}}  e^{-\frac{1}{12}s^4-\frac{\lambda}{2}s^2}\: \mathrm{d}s.
\end{equs}
The second part of the statement follows from the monotonicity of $d_n(\lambda)$ as a function of $\lambda$.
\end{proof}

\end{document}